\documentclass[11pt]{amsart}
\usepackage{amsmath,amssymb,mathrsfs,bm,enumerate}
\usepackage{pstricks}
\usepackage{mathtools} 
\usepackage[pagebackref, colorlinks=true, linkcolor=black, citecolor=black, urlcolor=black]{hyperref}
\usepackage{dsfont} 
\usepackage{tikz-cd} 
\tikzcdset{every label/.append style = {font = \small}} 

\usepackage[enableskew]{youngtab} 

\usepackage[boxsize=0.7em]{ytableau} 

\usepackage{stmaryrd} 

\newcommand{\excise}[1]{}

\renewcommand{\setminus}{\smallsetminus}

\renewcommand{\bar}{\overline}

\newcommand\mdoubleplus{\ensuremath{\mathbin{+\mkern-10mu+}}}

\newcommand{\bGr}{\mathbf{Gr}}

\newcommand{\bt}{\bm{\tau}}
\newcommand{\btp}{\bm{\tau^+}}

\newcommand{\bp}{\bm{p}}
\newcommand{\bq}{\bm{q}}

\DeclareMathOperator{\codim}{codim}

\newtheorem{thm}{Theorem}

\newtheorem{cor}[thm]{Corollary}

\newtheorem{theorem}{Theorem}[section]
\newtheorem{lemma}[theorem]{Lemma}
\newtheorem{proposition}[theorem]{Proposition}
\newtheorem{corollary}[theorem]{Corollary}

\newtheorem*{thm*}{Theorem}
\newtheorem*{lem*}{Lemma}
\newtheorem*{prop*}{Proposition}
\newtheorem*{cor*}{Corollary}

\theoremstyle{definition}

\newtheorem*{defn*}{Definition}
\newtheorem*{rmk*}{Remark}

\theoremstyle{definition}
\newtheorem{definition}[theorem]{Definition}
\theoremstyle{remark}
\newtheorem{remark}[theorem]{Remark}
\newtheorem{example}[theorem]{Example}

\makeatletter
\@namedef{subjclassname@2020}{%
\textup{MSC2020}}
\makeatother

\begin{document}

\title[Motivic classes of  degeneracy loci and pointed Brill-Noether]{Motivic classes of  degeneracy loci \\
and pointed Brill-Noether varieties}

\author[D. ~Anderson]{Dave Anderson}
\address{Dave Anderson
\newline \indent Department of Mathematics 
\newline \indent The Ohio State University,  Columbus, OH 43210}
\email{anderson.2804@math.osu.edu}
\author[L.~Chen]{Linda Chen}
\address{Linda Chen
\newline \indent Department of Mathematics and Statistics
\newline \indent Swarthmore College,  Swarthmore, PA 19081}
\email{lchen@swarthmore.edu}
\author[N.~Tarasca]{Nicola Tarasca}
\address{Nicola Tarasca 
\newline \indent Department of Mathematics \& Applied Mathematics
\newline \indent Virginia Commonwealth University, Richmond, VA 23284}
\email{tarascan@vcu.edu}


\begin{abstract}
Motivic Chern  and  Hirzebruch classes are polynomials with K-theory and  homology classes as coefficients,  which specialize to  Chern-Schwartz-\-MacPherson classes,  K-theory classes, and Cappell-Shaneson L-classes.
We provide formulas to compute the motivic Chern and Hirzebruch  classes of  Grassmannian and vexillary degeneracy loci. 
We apply our results to obtain the Hirzebruch $\chi_y$-genus of classical and one-pointed Brill-Noether varieties, and therefore their topological Euler characteristic, holomorphic Euler characteristic, and signature.
\end{abstract}

\maketitle


The study of motivic Chern classes unifies several theories of characteristic classes of singular varieties \cite{bsy}.
The motivic Chern and Hirzebruch classes are polynomials in a formal variable $y$; the motivic Hirzebruch class $T_y$ specializes to the Chern-Schwartz-MacPherson (CSM) class for $y=-1$, the K-theoretic Todd class for $y=0$, and the Cappell-Shaneson L-class  for $y=1$. The top degree term of $T_y$ of a compact variety gives the  Hirzebruch $\chi_y$-genus which specializes to the topological Euler characteristic  for $y=-1$, the holomorphic Euler characteristic  for $y=0$, and the signature for $y=1$. 

In this article, we compute the motivic Chern and Hirzebruch classes of Grassmannian and vexillary degeneracy loci  in type A. In particular, our results give formulas for their CSM classes and L-classes.
Several invariants of these degeneracy loci have been computed, e.g., a determinantal formula for their classes in cohomology   \cite{MR1154177, af1}, in K-theory   \cite{himn, anderson2019k}, and in algebraic cobordism  \cite{hudson2018vexillary}.  In the important special case of  a {degeneracy locus}   of a single map between vector bundles, the CSM class was computed by Parusi\'nski-Pragacz  \cite{pp}.  However, formulas for their L-classes and for CSM classes of more general degeneracy loci were not known. CSM classes and motivic Chern classes have  recently been studied for Schubert varieties and Schubert cells in flag manifolds \cite{aluffi2009chern, huh2016positivity, aluffi2016chern,  aluffi2017shadows, aluffi2019motivic} and matrix Schubert cells \cite{feher2018chern, zhang2018chern, frw}.
Hirzebruch $\chi_y$-genera have also been computed in other instances, e.g., for Hilbert schemes of points \cite{gottsche1993perverse,cappell2013characteristic} and for singular toric varieties~\cite{maxim2015characteristic}.

We consider maps of vector bundles over a smooth algebraic variety $X$:
\begin{equation*}
E_{p} \xrightarrow{\varphi} F_{q_1} \twoheadrightarrow F_{q_2} \dots \twoheadrightarrow F_{q_t}
\end{equation*}
with ${\rm rank}\left(E_p\right)=p$ and ${\rm rank}\left(F_{q_i}\right)=q_i$.  The  \textit{Grassmannian degeneracy locus} corresponding to the partition $\bm{\lambda} = (\lambda_1\geq \cdots \geq \lambda_t\geq 0)$ with $\lambda_i:=q_i-p+i$ is defined as
\[
W_{\bm{\lambda}}:=\left\{x\in X \, : \, \dim \textrm{ker}\left(E_{p} \rightarrow F_{q_i}\right)|_x \geq i \right\}.
\]
 More generally, we will consider maps of vector bundles 
\begin{equation}
\label{eq:maps}
E_{p_1}\hookrightarrow E_{p_2} \dots \hookrightarrow E_{p_t} \xrightarrow{\varphi} F_{q_1} \twoheadrightarrow F_{q_2} \dots \twoheadrightarrow F_{q_t}
\end{equation}
over a smooth algebraic variety $X$, with ${\rm rank}\left(E_{p_i}\right)=p_i$ and ${\rm rank}\left(F_{q_i}\right) =q_i$. Note that
 $0<p_1\leq \cdots \leq p_t$ and $q_1\geq \cdots\geq q_t>0$.
 Given a (weakly) increasing sequence $\bm{k}:=(k_1,\dots, k_t)$ of positive integers, the  \textit{vexillary degeneracy locus} corresponding to the triple $\bt = (\bm{k},\bm{p},\bm{q})$ is defined as
\[
W_{\bt}:=\left\{x\in X \, : \, \dim \textrm{ker}\left(E_{p_i} \rightarrow F_{q_i}\right)|_x \geq k_i \right\}.
\]
This nomenclature arises because the rank conditions can be described by the \textit{vexillary permutations} from \cite{lascoux1982geometrie} (i.e., permutations avoiding the pattern $2\,1\,4\,3$), see \cite[\S 1]{anderson2012degeneracy}.  The Grassmannian case is recovered when $p_1=\cdots =p_t$ and $k_i=i$ for each $i$.

Our main result gives formulas to compute the motivic Chern and Hirzebruch classes of vexillary degeneracy loci.
We proceed in two steps. 

First, we relate the motivic Hirzebruch classes of a vexillary degeneracy locus $W_{\bt}$ and a certain resolution of $W_{\bt}$. As in Kempf-Laksov \cite{laksov1974determinantal}, $W_{\bt}$ is resolved by $\phi\colon \widetilde{\Omega}_{\bm{\tau}}\rightarrow W_{\bt}$ where $\widetilde{\Omega}_{\bm{\tau}}$ is the variety parametrizing complete flags of sub-bundles $V_{1}\subseteq \cdots\subseteq V_{k_t}$ such that $\mathrm{rank}\left( V_{i}\right)=i$ and $V_{k_i}\subseteq \mathrm{ker}\left(E_{p_i}\rightarrow F_{q_i} \right)$ for each $i$,  see \S\ref{sec:motref}. (In \S\ref{sec:Omega} we will also study a partial resolution $\widetilde{\Omega}\rightarrow \Omega \rightarrow W$, hence the use of the tilde here.)
Let $\iota \colon W_{\bm{\tau}} \hookrightarrow X$ denote the inclusion.
We carry out an explicit computation of the class $(\iota\phi)_*\,T_y\left(\widetilde{\Omega}_{\bm{\tau}} \right)$:

\begin{thm}
\label{thm:secondintrothm}
For a triple $\bt=(\bm{k}, \bm{p}, \bm{q})$ and with assumptions as in \S\ref{sec:assumptions}, the class $(\iota\phi) _*\,T_y\left(\widetilde{\Omega}_{\bm{\tau}} \right)$ is  computed by a universal operator applied to $\left[ W_{\bt}\right]\cap T_y(X)$, where $\left[ W_{\bt}\right]$ is the determinantal formula for the class of $W_{\bt}$ in $A^*(X)$  (explicitly,  Theorem \ref{thm:TyOmegatilde}).
\end{thm}

Second, since the fibers of $\phi$ are not constant, in order to compare the motivic Hirzebruch classes of $W_{\bt}$ and $\widetilde{\Omega}_{\bm{\tau}}$, we proceed to find the 
stratification of $W_{\bt}$ into locally closed strata on which $\phi$ is  locally trivial. 

As reviewed in \S\ref{sec:tau'} after \cite{af1}, a triple $\bt=(\bm{k}, \bm{p}, \bm{q})$ can be inflated to a triple $\bt'=(\bm{k}', \bm{p}', \bm{q}')$ by inflating the sequences $\bm{k}$, $\bm{p}$, and $\bm{q}$ of length $t$ to sequences $\bm{k}'$, $\bm{p}'$, and $\bm{q}'$ of length $k_t$ with $\bm{k}'=(1,2,\dots,k_t)$ such that 
\[
q'_1-p'_1+1 \,>\, q'_2-p'_2+2 \,>\, \cdots \,>\, q'_{k_t}-p'_{k_t}+k_t \,>0
\]
and $W_{\bt}=W_{\bt'}$.
The locus $W_{\bt}$ contains the loci  $W_{\btp}$ for $\btp=\left(\bm{k^+}, \bm{p}', \bm{q}'\right)$ with $\bm{k^+}\geq\bm{k}'$ in componentwise order. For degree reasons, there are only finitely many such sub-loci $W_{\btp}\subseteq W_{\bt}$. 
The map $\phi$ is locally trivial  precisely on the locally closed strata $W_{\btp}^\circ \subset W_{\btp}$ (defined in \eqref{eq:locclosedstrata} as expected) --- see \S\ref{sec:filtr}.
We express the class $(\iota\phi)_*\,T_y\left(\widetilde{\Omega}_{\bm{\tau}} \right)$ computed in Theorem \ref{thm:secondintrothm} in terms of the motivic Hirzebruch classes of the strata of $W_{\bt}$:

\begin{thm}
\label{thm:firststep}
For a triple $\bt=(\bm{k}, \bm{p}, \bm{q})$ and with assumptions as in \S\ref{sec:assumptions}, one has
\[
(\iota\phi)_* \, T_y\left(\widetilde{\Omega}_{\bt}\right) = \sum_{\bm{k^+}} (-y)^{|\bm{k^+}|-|\bm{k}'|} \,\,  \iota_*\,T_y\left(W_{\btp}\right)
\]
where $\btp=\left(\bm{k^+}, \bm{p}', \bm{q}'\right)$ and the sum is over the set of weakly increasing sequences $\bm{k^+}\geq\bm{k}'=(1,\dots, k_t)$. 
\end{thm}

For instance, consider the triples $\bt(i):=((i),(3),(3))$, for $i=1,2,3$, and the triple $\bt(22):=((2,2),(2,3),(3,3))$. The strata required for the locus $W_{\bt(1)}$ are $W_{\bt(i)}^\circ$ for $i=1,2,3$. Since the triple $\bt(2)$ is inflated to the triple $((1,2),(2,3),(3,3))$, the strata required for the locus $W_{\bt(2)}$ are $W_{\bt(2)}^\circ$, $W_{\bt(3)}^\circ$, and $W_{\bt(22)}^\circ$.
Then Theorem \ref{thm:firststep} gives
\begin{align*}
(\iota\phi) _*\,T_y\left(\widetilde{\Omega}_{\bt(1)} \right) &= \iota_*\,T_y \left(W_{\bt(1)}\right) -y\, \iota_*\,T_y \left(W_{\bt(2)}\right) +y^2 \, \iota_*\,T_y \left(W_{\bt(3)}\right),\\
(\iota\phi) _*\,T_y\left(\widetilde{\Omega}_{\bt(2)} \right) &= \iota_*\,T_y \left(W_{\bt(2)}\right) -y \,\iota_*\,T_y \left(W_{\bt(22)}\right) -y \,\iota_*\,T_y \left(W_{\bt(3)}\right),\\
(\iota\phi) _*\,T_y\left(\widetilde{\Omega}_{\bt(3)} \right) &= \iota_*\,T_y \left(W_{\bt(3)}\right),\\
(\iota\phi) _*\,T_y\left(\widetilde{\Omega}_{\bt(22)} \right) &= \iota_*\,T_y \left(W_{\bt(22)}\right).
\end{align*}

Finally, by the inclusion-exclusion principle, applying Theorem \ref{thm:firststep} to the (closure of the) strata of $W_{\bt}$, and then to the strata of the strata, and so on, one can express the motivic Hirzebruch class of $W_{\bt}$ in terms of classes $(\iota\phi)_* \, T_y\left(\widetilde{\Omega}_{\bm{\sigma}}\right)$ corresponding to a subset of  the finitely many vexillary degeneracy sub-loci $W_{\bm{\sigma}}\subseteq W_{\bt}$. Thus we have:

\begin{thm}
\label{thm:final}
Combining Theorems~\ref{thm:secondintrothm} and \ref{thm:firststep} allows one to compute the motivic Hirzebruch class of~$W_{\bt}$ for any triple $\bt$.
\end{thm}

As an example, for the triple $\bt(1):=((1),(3),(3))$ as above, solving for $\iota_*\,T_y \left(W_{\bt(1)}\right)$ gives
\begin{equation*}
\iota_*\,T_y \left(W_{\bt(1)}\right) = (\iota\phi) _*\,T_y\left(\widetilde{\Omega}_{\bt(1)} \right) +y \, (\iota\phi) _*\,T_y\left(\widetilde{\Omega}_{\bt(2)} \right) + y^2\, (\iota\phi) _*\,T_y\left(\widetilde{\Omega}_{\bt(22)} \right).
\end{equation*}
Theorem \ref{thm:secondintrothm} can then be applied to compute the right-hand side as a polynomial in $y$ with coefficients expressed in terms of the Chern classes of the given vector bundles.

\subsection*{Brill-Noether theory} 
In \S\ref{sec:pBNvar} we apply our results to classical and pointed Brill-Noether varieties.
Brill-Noether theory studies the geometry of line bundles and linear series on algebraic curves.  For a smooth algebraic curve $C$, the \textit{classical Brill-Noether variety} $W^r_d(C)$ parametrizes  line bundles of degree $d$ on $C$ having at least $r+1$ independent global sections \cite{MR770932}.  More generally,  the \textit{pointed Brill-Noether variety} $W_d^{\bm{a}}(C,P)$ parametrizes  line bundles of degree $d$ on $C$ having at least $r+1$ independent global sections with vanishing orders  at the  point $P$ at least equal to $\bm{a}=(0\leq a_0 < \cdots < a_r \leq d)$. In \S \ref{sec:motrefBN}, we  compute the motivic Hirzebruch class of  $W^r_d(C)$ and $W_d^{\bm{a}}(C,P)$ for a general $(C,P)$, as these are examples of Grassmannian degeneracy loci. This extends the study of the CSM class of the classical Brill-Noether varieties $W^r_d(C)$ treated by Parusi\'nski-Pragacz~\cite{pp}.

Similarly, we compute the motivic Hirzebruch class of the \textit{Brill-Noether variety} $G^r_d(C)$ parametrizing linear series on $C$ of degree $d$ and projective dimension $r$, and its pointed counterpart $G^{\bm{a}}_d(C,P)$ parametrizing linear series on $C$ of degree $d$ with prescribed vanishing $\bm{a}$ at the point $P$. For these, we use the result presented more generally about a degeneracy locus $\Omega_{\bm{\lambda}}$ in~\S\ref{sec:Omega}.

For a smooth \textit{Brill-Noether-Petri general} curve $C$ of genus $g$, one has: (i) $G^r_d(C)$  is smooth and has dimension equal to $\rho(g,r,d):=g-(r+1)(g-d+r)$; (ii)  $W^r_d(C)$ has dimension equal to $\rho(g,r,d)$, provided that $g-d+r\geq 0$; and (iii) when $g-d+r>0$, the singular locus of $W^r_d(C)$ coincides with $W^{r+1}_d(C)\subset W^r_d(C)$ \cite[pg.~214]{MR770932}. 

For instance, $W^r_d(C)$ is smooth when $C$ is a smooth {Brill-Noether-Petri general} curve, $\rho(g,r,d)\leq 2$ and $g\geq 2$. Write $\lambda=g-d+r$, and let $\theta$ be the cohomology class of the theta divisor in $\mathrm{Pic}^d(C)$. In the surface case, we prove:

\begin{cor}
\label{cor:TyWrd}
Fix $g\geq 2$ and $r,d$ such that $\rho(g,r,d)=2$. For a Brill-Noether-Petri general smooth curve $C$ of genus $g$, the motivic Hirzebruch class of the surface $W^r_d(C)$~is
\begin{multline*}
T_y\left( W^r_d(C) \right) =\bigg( 1+  \frac{\lambda(r+3)(y-1)}{2(\lambda+2)}\,\theta \\
{}+ \frac{\lambda(r+1)\left(\lambda(r+1)(y-1)^2-2y\right)}{ 2(\lambda+r)(\lambda+r+2)} \theta ^2\bigg)
\cdot \theta^{g-2} \prod_{i=0}^r \frac{i!}{(g-d+r+i)!}
\end{multline*}
in $H^*\left(\mathrm{Pic}^d(C) \right)[y]$. 
Since the top degree term of $T_y$ gives the Hirzebruch $\chi_y$-genus, one has
\[
\chi_y\left( W^r_d(C) \right) = g! \frac{\lambda(r+1)\left(\lambda(r+1)(y-1)^2-2y\right)}{ 2(\lambda+r)(\lambda+r+2)}  \prod_{i=0}^r \frac{i!}{(g-d+r+i)!}.
\]
Moreover, with the same hypotheses, one has $\chi_y\left( G^r_d(C) \right)=\chi_y\left( W^r_d(C) \right)$.
\end{cor}

The formula for $\chi_y$ in Corollary \ref{cor:TyWrd} recovers the topological and holomorphic Euler characteristics for $y=-1$ and $y=0$ (known from \cite{pp} and \cite{act}, respectively), and gives a new result on the \textit{signature} of the surface $W^r_d(C)$ for $y=1$:
\[
\sigma\left( W^r_d(C) \right) = g! \frac{2-g}{(g-d+2r)(g-d+2r+2)}\prod_{i=0}^r \frac{i!}{(g-d+r+i)!}
\]
and similarly for the \textit{surface} $G^r_d(C)$, since in this case $\sigma\left( G^r_d(C) \right)=\sigma\left( W^r_d(C) \right)$.

We briefly review the definition of the signature.
For a compact oriented manifold $X$ of real dimension $4k$, consider the   non-degenerate symmetric bilinear form on the finite-dimensional vector space $H^{2k}(X,\mathbb{R})$ given by
\[
\langle \alpha, \beta \rangle := \int_X \alpha \cup \beta, \qquad \mbox{for $\alpha,\beta \in H^{2k}(X,\mathbb{R})$.}
\]
The signature $\sigma(X)$ of $X$ is defined as the number of positive entries minus the number of negative entries in a diagonalized version of this form. When $X$ is smooth, Hirzebruch's signature theorem expresses $\sigma(X)$ as a universal linear combination of the Pontrjagin numbers of the tangent bundle of $X$ \cite{hirzebruch1966topological}. When $X$ is possibly singular, $\sigma(X)$ is computed by the top degree term of the L-class of $X$ \cite{cappell1991stratifiable, bsy}.

When $g=2$, then $W^r_d(C)=\mathrm{Pic}^d(C)$, thus indeed  $\chi_y=0$, since Abelian varieties have trivial tangent bundle. As a further check, $\sigma$ and the topological Euler characteristic $\chi_{\mathrm{top}}$ always have the same parity \cite[Cor.~64]{rodriguez2005signature}, and this is indeed satisfied by the formulas resulting from Corollary \ref{cor:TyWrd}. 
In fact, we observe the following, perhaps surprising, relations for the surface~$W^r_d(C)$:
\begin{align*}
\frac{g-2}{2}\, \sigma\left( W^r_d(C) \right) &= -\chi_{\mathrm{hol}}\left( W^r_d(C) \right),\\
(2g-3)\, \sigma\left( W^r_d(C) \right) &= -\chi_{\mathrm{top}}\left( W^r_d(C) \right),\\
(g-2)\, \chi_{\mathrm{top}}\left( W^r_d(C)\right) &=(4g-6)\, \chi_{\mathrm{hol}}\left( W^r_d(C)\right).
\end{align*}

In \S\ref{sec:ptBNsurfpencils} we also compute explicitly the motivic Hirzebruch class of pointed Brill-Noether surfaces parametrizing pencils (i.e., $r=1$).
Interestingly, we show that the motivic Hirzebruch classes  of Grassmannian degeneracy loci  corresponding to a partition $\bm{\lambda}$ do not specialize from the case $\lambda_i>\lambda_{i+1}$ to the case $\lambda_i=\lambda_{i+1}$, for some $i$ (Remark \ref{rmk:piecepol}). 
This is in contrast with the  K-theory class (the case $y=0$) of  degeneracy loci with rank conditions imposed by arbitrary (not only vexillary) permutations, given by Grothendieck polynomials \cite{fl}.

\subsection*{Our strategy} 
Our approach draws a great deal of inspiration from Parusi\'nski-Pragacz \cite{pp}. Indeed, all computations build on three fundamental computations of CSM classes, treated in~\cite{pp}: 
\begin{enumerate}[(a)]
\item  the CSM class of the zero locus of a regular section of a vector bundle; 
\item  the CSM class of a Grassmannian bundle; and 
\item  the push-forward of the CSM class via a fibration.
\end{enumerate}
In  \S \ref{sec:motivic}, we compute the \textit{motivic refinement} of  (a), (b), (c), see Lemmata \ref{lem:zeros}, \ref{lem:Gr}, \ref{lemma:Fibra}. The CSM classes in cases (a), (b), (c)  are recovered by specializing to $y=-1$. 
In order to make the generalization to Grassmannian and vexillary degeneracy loci possible, we combine (a), (b), (c) with the  proof strategy introduced in \cite{af1} (see \S\ref{sec:motref}). Additional care is needed in dealing with motivic Chern and Hirzebruch classes when one generalizes methods designed for fundamental classes. Indeed, while the fundamental class coincides with the push-forward of the fundamental class of a resolution, motivic Chern and Hirzebruch classes are more delicate to handle. 
To arrive at the motivic Chern class of a vexillary degeneracy locus $W_{\bt}$, we first compute the push-forward of the motivic Chern  class of a resolution $\widetilde{\Omega}_{\bt}$ of $W_{\bt}$ (Theorem \ref{thm:TyOmegatilde}). 
Furthermore, using (c), we relate the push-forward of the motivic Chern class of $\widetilde{\Omega}_{\bt}$ with $\bt=(\bm{k}, \bm{p}, \bm{q})$ to the motivic   Chern class of varieties $W_{\btp}$ with $\btp=\left(\bm{k^+}, \bm{p}', \bm{q}'\right)$, for $\bm{k^+}\geq \bm{k}'$ (Theorem \ref{thm:1Omegatildefrom1W}). Finally, 
the motivic Chern class of $W_{\bt}$ follows by the inclusion-exclusion principle.

A key step in our argument is the careful analysis in \S\ref{sec:filtr} of the stratification of a  degeneracy locus induced from its resolution. We stratify $W_{\bt}$  with $\bt=(\bm{k}, \bm{p}, \bm{q})$ by the  loci $W_{\btp}^\circ$ with $\btp=\left(\bm{k^+}, \bm{p}', \bm{q}'\right)$, for $\bm{k^+}\geq \bm{k}'$. 
These are precisely the strata on which  the resolution $\widetilde{\Omega}_{\bt}\rightarrow W_{\bt}$ is  locally trivial. 

Theorem \ref{thm:1Omegatildefrom1W} uses the \textit{additivity} of the motivic Hirzebruch class $T_y$ as a transformation from the Grothendieck group $K_0(\mathrm{var}/X)$ of algebraic varieties over $X$.
In fact, the Hirzebruch $\chi_y$-genus is \textit{the most general additive genus} \cite{bsy}. Computation of other invariants, as the elliptic class and elliptic genus, would thus require new strategies.

\smallskip

Specializing to the case when the ambient variety $X$ is a Grassmannian, our results give the motivic class of its Schubert varieties in terms of the motivic class of the Grassmannian. 
For instance, when $y=-1$, using the formula of  \cite{aluffi2009chern} for the CSM class of Grassmannians, one can verify that the resulting formulas for the CSM class of Schubert varieties are consistent with the results in \cite{aluffi2009chern} after some nontrivial combinatorics (see~\S\ref{sec:SchubinGr}).

Motivic Chern classes of Schubert cells in partial flag varieties have  been computed in \cite{frw} via localization; the classes of the Schubert varieties could then be obtained from \cite{frw} by summing over all the strata of the closure of the Schubert cells. 
In the case of vexillary permutations, our strategy produces  the classes of Schubert varieties as a first outcome; this makes it feasible to arrive at a viable formula in the application to the Brill-Noether setting (see \S \ref{sec:motrefBN}).

\subsection*{Open questions}
In \cite{act} we study more generally two-pointed Brill-Noether varieties, and show that they have the structure of determinantal varieties obtained from maps of flag bundles with rank conditions imposed by 321-avoiding permutations. We compute in \cite{act} their connective K-theory class and holomorphic Euler characteristic. It would be interesting  to compute their  motivic Chern class, extending this work to the two-pointed case. 

The holomorphic Euler characteristic of two-pointed Brill-Noether varieties is expressed in \cite{chan2017euler} as the enumeration of certain standard set-valued tableaux. 
We have found tableau formulas expressing the Hirzebruch $\chi_y$-genus  of one-pointed Brill-Noether surfaces.
We wonder whether there exist tableau formulas expressing the Hirzebruch $\chi_y$-genus  of one- or two-pointed Brill-Noether varieties in general.

\smallskip

\noindent {\it Acknowledgements.} We are indebted to \cite{MR1154177, af1} for the treatment of vexillary degeneracy loci,  to \cite{pp} for the treatment of CSM classes and degeneracy loci, and to \cite{bsy} for inspiring us to consider  motivic Chern and Hirzebruch classes. We would like to thank J\"org Sch\"urmann for clarifying some points in an earlier version of \S\ref{sec:motivic}, and the referee for a careful reading.


\section{Vexillary degeneracy loci}

Here we set  the notation and collect the assumptions used throughout.
We adopt  the notation of triples $\bt$ from \cite{af1, anderson2019k}.

\subsection{Vexillary degeneracy loci}
\label{sec:setup}
Let $X$ be an irreducible variety over an algebraically closed field.
Given maps of vector bundles over  $X$ as in \eqref{eq:maps}
and a (weakly) increasing sequence $\bm{k}:=(k_1,\dots, k_t)$  of positive integers, the  \textit{vexillary degeneracy locus} corresponding to the triple $\bm{\tau} = (\bm{k},\bm{p},\bm{q})$ is defined as
\[
W_{\bm{\tau}}:=\left\{x\in X \, : \, \dim \textrm{ker}\left(E_{p_i} \rightarrow F_{q_i}\right)|_x \geq k_i  \, \mbox{ for all $i$}\right\}
\]
with inclusion $\iota\colon W_{\bt} \hookrightarrow X$. 
Such a locus is Cohen-Macaulay  when it has the expected dimension and when $X$ is Cohen-Macaulay.
We will compute the motivic Hirzebruch class of $W_{\bt}$ in $A_*(X)[y]$, where $A_*(X)$ is the Chow group of $X$ and $y$ is a formal variable.
It will be convenient to consider  arbitrary weakly increasing sequences $\bm{k}$ (see e.g., \S\ref{sec:filtr}).

\subsubsection{The reduced triple $\bar{\bt}$} 
\label{sec:kpqbar}
Some of the conditions defining $W_{\bt}$ may be redundant, and $W_{\bt}$ could be similarly described as the vexillary degeneracy locus corresponding to a triple consisting of shorter sequences. 
After \cite{MR1154177}, a triple $\bt=(\bm{k}, \bm{p}, \bm{q})$ with $\bm{p}=(0\leq p_1\leq \cdots \leq p_t)$ and $\bm{q}=(q_1\geq \cdots \geq q_t\geq 0)$ is called \textit{essential} if 
\[
0< k_1 < \cdots < k_t \quad\mbox{and}\quad q_1-p_1+k_1 > \cdots > q_t-p_t+k_t >0.
\]

Given $\bt=\left({\bm{k}},{\bp},{\bq}\right)$, we denote by $\bar{\bt}=\left(\bar{\bm{k}},\bar{\bp},\bar{\bq}\right)$  the essential triple of  \textit{shortest subsequences} $\bar{\bm{k}}$, $\bar{\bm{p}}$, and $\bar{\bm{q}}$ of $\bm{k}$, $\bm{p}$, and $\bm{q}$ such that  $W_{\bt}= W_{\bar{\bt}}$.
Necessarily, $\bar{\bm{k}}$ is {strictly} increasing.

\begin{example}
The sequences $\bm{k}=(2,2,3,4)$, $\bm{p}=(4,5,6,7)$, and $\bm{q}=(8,7,6,3)$ and their subsequences $\bar{\bm{k}}=(2,3)$, $\bar{\bm{p}}=(4,6)$, and $\bar{\bm{q}}=(8,6)$ describe the same degeneracy locus.
\end{example}

\subsubsection{The inflated triple ${\bt}'$} 
\label{sec:tau'}
Assume the flag $E_{p_1}\subseteq E_{p_2}\subseteq \cdots \subseteq E_{p_t}$ extends to a full flag of sub-bundles $E_{1}\subseteq E_2\subseteq \cdots \subseteq E_{p_t}$ defined on $X$, and similarly, the flag $F_{q_1}\twoheadrightarrow F_{q_2}\twoheadrightarrow \cdots \twoheadrightarrow F_{q_t}$ extends to a full flag of quotients $F_{q_1}\twoheadrightarrow \cdots \twoheadrightarrow F_2\twoheadrightarrow   F_{1}$ defined on $X$. This assumption is indeed not restrictive, see Remark \ref{rmk:fullflags}.
Contrary to the previous subsection, we describe here how to inflate the triple $\bt$ to a triple consisting of \textit{longer} sequences defining the same locus $W_{\bt}$. This will be used  to define the locus $\widetilde\Omega_{\bt}$ in \S\ref{sec:Omegatildedef}.
 
Assume the triple $\bt$ is essential, that is, $\bt=\bar{\bt}$.
The triple $\bt$ can then be inflated to a triple $\bt'=(\bm{k}',\bp',\bq')$ by inflating the sequences $\bm{k}, \bm{p}$, and $\bm{q}$ of length $t$ to  sequences $\bm{k}', \bm{p}'$, and $\bm{q}'$ of length $k_t$ with $\bm{k}'=(1,2,\dots, k_t)$ as in \cite[\S 1.4]{af1} (see also further details in \cite[\S 1]{anderson2019k}). 
Namely, suppose that $k_i> k_{i-1}+1$, for some~$i$. Then necessarily  $p_i>p_{i-1}$ or $q_i<q_{i-1}$ (otherwise the corresponding conditions defining $W_{\bt}$ are redundant, and the triple $\bt$ can be reduced).
When $p_i>p_{i-1}$, inflate by inserting the entry $k_i-1$ between $k_{i-1}$ and $k_i$ in $\bm{k}$, the entry $p_i-1$ between $p_{i-1}$ and $p_i$ in $\bm{p}$, and the entry $q_i$ between $q_{i-1}$ and $q_i$ in $\bm{q}$.
On the geometric side, the condition $\dim \mathrm{ker}\left( E_{p_i}\rightarrow F_{q_i}\right)|_x \geq k_i$  implies  $\dim \mathrm{ker}\left( E_{p_i-1}\rightarrow F_{q_i}\right)|_x \geq k_i-1$ for a point $x\in W_{\bt}$. 
The case $p_i=p_{i-1}$ and $q_i<q_{i-1}$ is treated similarly. Proceeding in this way, one arrives at sequences $\bm{k}', \bm{p}'$, and $\bm{q}'$ of length $k_t$ such that $\bm{k}'=(1,2,\dots, k_t)$ and $W_{\bt}= W_{\bt'}$.

\subsubsection{Feasibility} 
\label{sec:lafeasible}
For an essential triple $\bt$, in order for the conditions defining the locus $W_{\bt}$ to be feasible, 
we assume that the sequence 
\begin{equation}
\label{lambda}
\lambda_{k_i} := q_i-p_i+k_i \quad\quad\mbox{for $i=1,\dots, t$,}
\end{equation}
is weakly decreasing.

\subsection{The partition \texorpdfstring{$\bm{\lambda}_{\bt}$}{lambda}}
\label{sec:la}
Fix an essential triple $\bt$.
Extending \eqref{lambda}, define the partition $\bm{\lambda}_{\bt}=(\lambda_1, \dots, \lambda_{k_t})$ as 
\[
\lambda_i:=\lambda_{k_a}\quad \mbox{for $k_{a-1}<i\leq k_{a}$.}
\]
For a triple $\bt$ which is not necessarily essential,  define $\bm{\lambda}_{\bt}:=\bm{\lambda}_{\bar{\bt}}$, where $\bar{\bt}$ is the reduced triple consisting of the shortest subsequences of $\bm{k}$, $\bm{p}$, $\bm{q}$ such that $W_{\bt}= W_{\bar{\bt}}$ (\S \ref{sec:kpqbar}).
The expected codimension of the locus $W_{\bt}$ in $X$ is 
\[
\codim_X\left(W_{\bt}\right) = \left|\bm{\lambda}_{\bt}\right|:=\sum_{i=1}^{k_t} \lambda_{i}.
\]

\begin{example}
For the triple $\bt=(\bm{k},\bp,\bq)$ where  $\bm{k}=(2,2,3)$, $\bm{p}=(4,5,6)$, and $\bm{q}=(8,7,6)$, one has $\bm{\lambda}_{\bt}=(6,6,3)$.
\end{example}

\subsection{The bundles $E(i)$, $F(i)$,  classes $c(i)$, and operators $T_y(i)$}
\label{sec:cidef}
For an essential triple $\bt$, 
define  the vector bundles $E(i)$ and $F(i)$ as
\[
E(i):= E_{p_a} \quad \mbox{and}\quad F(i):= F_{q_a} \qquad \mbox{for $k_{a-1}<i\leq k_a$.}
\]
For a triple $\bt$ which is not necessarily essential,  define 
\[
\quad E(i) := E_{\bar{p}_a} \quad \mbox{and}\quad
F(i) := F_{\bar{q}_a}
\quad \mbox{for $\bar{k}_{a-1}<i\leq \bar{k}_a$.}
\]
Here $\bar{\bm{k}},\bar{\bm{p}},\bar{\bm{q}}$ are the shortest subsequences of $\bm{k}$, $\bm{p}$, $\bm{q}$ such that $W_{\bt}= W_{\bar{\bt}}$  with $\bar{\bt} =\left(\bar{\bm{k}},\bar{\bp},\bar{\bq}\right)$ (see \S \ref{sec:kpqbar}). 
Let 
\[
c(i):=c\left(F(i)-E(i)\right) = \frac{c\left(F(i)\right)}{c\left(E(i)\right)} \,\in\, A^*(X).
\]

\subsubsection{}
\label{sec:Ri}
For each $i$, the \textit{raising operator} $R_i$ increases the index of the class $c(i)$ by one, that is, 
\[
R_i\,c(i)_m= c(i)_{m+1} \qquad \mbox{and} \qquad R_i \,c(k)_m=c(k)_m, \quad \mbox{for $k\not= i$.}
\]
Moreover, $R_i$ is extended linearly over $\mathbb{Q}$ and multiplicatively on monomials in  classes $c(j)$, e.g., $R_i \left( c(j)_m\, c(k)_n \right)= \left(R_i \, c(j)_m\right) \left(R_i \,c(k)_n\right)$.

\subsubsection{}
For a formal variable $R$ and a vector bundle $E$ of rank $e$ with Chern roots $a_i$, for $i=1,\dots, e$, define 
\[
T_y\left(R\otimes E\right):=\prod_{i=1}^{e} Q_y\left(R+ a_i\right), \,\,\mbox{with}\,\,
Q_y\left(\alpha\right):= \frac{\alpha(1+y)}{1-e^{-\alpha(1+y)}}-\alpha y  \in \mathbb{Q}[y]\llbracket \alpha \rrbracket.
\]
As in \eqref{eq:Qyexpansion}, the  terms in degree at most two are given by
\begin{eqnarray*}
T_y\left(R\otimes E\right) &=& 1 + \frac{1}{2}(1-y)\left(eR + c_1(E) \right) \\
&&{}+ \frac{1}{12}(1+y)^2 \left( eR^2 + 2 c_1(E)R + \mathrm{ch}_2(E) \right)\\
&&{}+\frac{1}{4}(1-y)^2 \left( {e\choose 2}R^2 + (e-1)c_1(E)R + c_2(E)\right) +\dots
\end{eqnarray*}
where $\mathrm{ch}(E)=\sum_{i=1}^e \mathrm{ch}_i(E)$ is the Chern character of $E$.
The motivic Hirzebruch class $T_y(X)$ of a smooth variety $X$ is recovered when $R=0$ and $E$ is the tangent bundle of $X$ (see \S\ref{sec:motivic}). 
Similarly, in the absence of $E$, we set
\[
T_y(R):=Q_y(R). 
\]

For each $i$, define the operator 
\[
T_y(i):=T_y\left( R_i \otimes \left( F(i)-E(i)\right)\right) 
\]
acting on the classes $c(j)$ such that $R_i$ acts as in \S\ref{sec:Ri}, and the Chern roots of the virtual bundle $F(i)-E(i)$ act by multiplication.

\subsection{Assumptions} 
\label{sec:assumptions}
We collect here the assumptions used throughout.
For a  triple $\bt=(\bm{k}, \bm{p}, \bm{q})$, consider the vexillary degeneracy locus $W_{\bt}$ and a sub-locus $W_{\btp}\subseteq W_{\bt}$ where $\btp=\left(\bm{k^+}, \bm{p}', \bm{q}'\right)$ with $\bm{k^+}\geq \bm{k}'$ in componentwise order.
The corresponding locally closed stratum is
\begin{equation}
\label{eq:locclosedstrata}
W_{\btp}^\circ:=\left\{x\in X \, : \, \dim \textrm{ker}\left(E_{p'_i} \rightarrow F_{q'_i}\right)\big|_x = k^+_i \right\}.
\end{equation}
This is 
\[
{W}_{\btp}^\circ = {W}_{\btp} \setminus \bigcup_{\bm{k^{\mdoubleplus}}>\bm{k}^+} {W}_{\bt^{\mdoubleplus}}\, \subseteq X, \qquad \mbox{where $\bt^{\mdoubleplus}=\left(\bm{k^{\mdoubleplus}}, \bm{p}', \bm{q}'\right)$.}
\]
We assume that $X$ is an irreducible smooth algebraic variety over an algebraically closed field of characteristic zero, and for all weakly increasing sequences $\bm{k^+}\geq \bm{k}'$ (i.e., \mbox{$k^+_i\geq i$} for all $i=1,\dots,k_t$),  the stratum ${W}_{\btp}^\circ$ with $\btp=\left(\bm{k^+}, \bm{p}', \bm{q}'\right)$ is  smooth of pure dimension $\dim X - |\bm{\lambda}_{\btp}|$.

\begin{remark}
In \cite{pp},  Parusi\'nski-Pragacz show that their formula for the CSM class of the degeneracy locus of a single map between vector bundles holds under a weaker assumption which allows one to consider   a \textit{possibly singular} analytic variety $X$ as  ambient variety. It seems reasonable to expect that  our results hold under a similar weaker assumption. 
For this, one  needs to upgrade Lemma \ref{lem:zeros} below, for instance by
 working in the analytic category with a Whitney stratification of a singular $X$ as in \cite{pp}. 
However, $X$ will be smooth in all the applications we consider.
\end{remark}


\section{Motivic classes: fundamental computations}
\label{sec:motivic}

After briefly reviewing motivic Chern and Hirzebruch classes following \cite{bsy}, we discuss here three fundamental computations: Lemmata \ref{lem:zeros}, \ref{lem:Gr}, and \ref{lemma:Fibra}. These will serve as the cornerstone of the paper.
 
 For an algebraic variety $X$ over a field of characteristic zero, let $K_0(\mbox{var}/X)$ be the Grothendieck group of algebraic varieties over $X$, let $K_*(X)$ be the Grothendieck group of coherent sheaves of $\mathscr{O}_X$-modules, and $A_*(X)$ the Chow group.
The transformations
\[
\begin{tikzcd}[row sep=2.5em]
 & K_0(\mbox{var}/X) \arrow{dl}[swap]{mC} \arrow{dr}{T_y} \\
K_*(X)\otimes \mathbb{Z}[y]  
&& A_*(X)\otimes \mathbb{Q}[y]
\end{tikzcd}
\]
are the unique transformations which commute with proper push-down and,  for $X$ smooth, satisfy 
\begin{align}
mC\left(\mbox{id}_X\right) &=\sum_{i\geq 0} \left[\wedge^i \,\mathscr{T}_X^\vee\right]y^i =:\lambda_y\left(\mathscr{T}_X^\vee\right), \label{eq:mC} \\
T_y\left(\mbox{id}_X\right) &= \prod_{i=1}^{\dim X} Q_y(\alpha_i) \cap [X] =:T_y\left(\mathscr{T}_X\right)\cap [X]. \label{eq:Ty}
\end{align}
Here, $\alpha_i$ are the Chern roots of the tangent bundle $\mathscr{T}_X$, and $Q_y\left(\alpha\right)$ is the series
\[
Q_y\left(\alpha\right):= \frac{\alpha(1+y)}{1-e^{-\alpha(1+y)}}-\alpha y  \qquad\in \mathbb{Q}[y]\llbracket \alpha \rrbracket
\]
starting as
\begin{equation}
\label{eq:Qyexpansion}
Q_y\left(\alpha\right) = 1 + \frac{1}{2}\alpha (1-y) +\frac{1}{12}\alpha^2 (1+y)^2 + \dots.
\end{equation}
The function $\lambda_y$ satisfies: $\lambda_y(a+b)=\lambda_y(a) \lambda_y(b)$. 
The function $T_y$ is the \textit{motivic Hirzebruch class} function introduced in \cite{hirzebruch1966topological}. One has $T_y(a+b)= T_y(a) T_y(b)$, as well.

For arbitrary $X$, let $\{X_i\}_{i\in I}$ be a stratification of $X$, with $X_i$ locally closed and smooth. By definition, we have
\begin{align}
\label{eq:singclasses}
mC\left({\rm id}_X\right) &:= \sum_{i\in I} mC\left(X_i \rightarrow X\right),  & T_y\left({\rm id}_X\right) &:= \sum_{i\in I} T_y\left(X_i \rightarrow X\right).
\end{align}
Since any two such stratifications admit a common refinement, the above is well defined.

\begin{remark}
The transformations $mC$ and $T_y$ satisfy $td_{(1+y)}\circ mC = T_y$, where
 \[
 td_{(1+y)}\colon K_*(X)\otimes \mathbb{Z}[y] \rightarrow A_*(X)\otimes \mathbb{Q}\left[y, (1+y)^{-1}\right]
 \]
is Yokura's generalization \cite{yokura1998singular} of the Todd class transformation from the singular Riemann-Roch theorem \cite{MR1644323}.
\end{remark}

\subsection{Zeros of sections}
\label{Zs}
For  a vector bundle $E$ on $X$,  let $\iota\colon Z\hookrightarrow X$ be the  zero locus of a regular section $s$
of $E\rightarrow X$. 

\begin{lemma}
\label{lem:zeros}
If $X$ is smooth and $s\colon X\rightarrow E$ meets transversally the zero section of $E$, then
\begin{align*}
mC\left(Z\hookrightarrow X\right) &= \frac{ \lambda_{-1}\left(E^\vee\right)}{ \lambda_y\left(E^\vee\right)}  mC\left({\rm id}_X\right),\\
T_y\left(Z\hookrightarrow X\right) &= \frac{ c_{\rm top}(E)}{ T_y(E)}  T_y\left({\rm id}_X\right).
\end{align*}
\end{lemma}

\begin{proof}
Since $s$ meets transversally the zero section of $E$, it follows that $Z$ is smooth.
We claim that
\begin{align*}
mC\left(Z\hookrightarrow X\right) = \frac{ \lambda_{-1}\left(E|_{X}^\vee\right)}{ \lambda_y\left(E|_{X}^\vee\right)}  mC\left({\rm id}_{X}\right).
\end{align*}
Indeed, we have
\begin{align*}
mC\left(Z\hookrightarrow X\right) &= \iota_! \, mC\left(\mbox{id}_{Z}\right) = \iota_! \, \lambda_y\left(\mathscr{T}^\vee_{Z}\right) =  \iota_! \left( \frac{\lambda_y\left(\mathscr{T}^\vee_{Z}\right)}{\iota^* \lambda_y\left(\mathscr{T}^\vee_{X}\right)} \right)  \lambda_y\left(\mathscr{T}^\vee_{X}\right) \\
&= \iota_! \, \left( \frac{1}{\iota^* \lambda_y\left(E|_{X}^\vee\right)} \right)  \lambda_y\left(\mathscr{T}^\vee_{X}\right) = \frac{ \lambda_{-1}\left(E|_{X}^\vee\right)}{ \lambda_y\left(E|_{X}^\vee\right)}  \lambda_y\left(\mathscr{T}^\vee_{X}\right).
\end{align*}
Here, $\iota_!$ is the K-theoretic push-forward via $\iota$. We have used the projection formula, and $\iota_!(1)= \lambda_{-1}(E|_{X}^\vee)$ in $K_*(X)$ (see for instance \cite[V, Prop.~4.3]{fula}; this also appears in \cite[\S 8.1]{frw}).
The statement for $mC$ follows. The same argument together with $\iota_*(1)=c_{\rm top}(E)\cap [X]$ in $A_*(X)$ prove  the statement for~$T_y$.
\end{proof}

When $y=-1$, Lemma \ref{lem:zeros} recovers the computation $\iota_*\, c_{\mathrm{SM}}(Z) = \frac{ c_{\rm top}(E)}{ c(E)} \cap  c_{\mathrm{SM}}(X)$ treated in
\cite[Proposition 1.3]{pp}.

\subsection {Grassmannian bundles} 
\label{Gr}
Given a vector bundle $E$ on $X$, let 
\[
\pi \colon {\rm Gr}(r,E)\rightarrow X
\]
 be the Grassmannian bundle parametrizing rank $r$ sub-bundles of $E$.
Consider the tautological exact sequence over ${\rm Gr}(r,E)$
\[
0\rightarrow S \rightarrow E \rightarrow Q \rightarrow 0.
\]

\begin{lemma}
\label{lem:Gr}
For arbitrary $X$, we have
\begin{align*}
mC\left({\rm id}_{{\rm Gr}(r,E)}\right) &= \lambda_y\left(\left(S^\vee\otimes Q\right)^\vee\right) \cdot \pi^* mC\left({\rm id}_X\right),\\
T_y\left({\rm id}_{{\rm Gr}(r,E)}\right) &= T_y\left(S^\vee\otimes Q\right) \cdot \pi^* \,T_y\left({\rm id}_X\right).
\end{align*}
\end{lemma}

\begin{proof}
When $X$ is smooth, ${\rm Gr}(r,E)$ is smooth.
From the definition \eqref{eq:mC}, we have
\[
mC\left(\mbox{id}_{{\rm Gr}(r,E)}\right) = \lambda_y\left(\mathscr{T}^\vee_{{\rm Gr}(r,E)}\right) = \lambda_y\left(\left(S^\vee\otimes Q\right)^\vee\right) \cdot \pi^* \lambda_y\left(\mathscr{T}^\vee _{X}\right),
\]
and the statement about the motivic Chern class follows. 
For arbitrary $X$, the statement about the motivic Chern class is an immediate application of the Verdier-Riemann-Roch formula from \cite[Cor.~2.1 (4)]{bsy}.
The proof for $T_y$ is similar using \cite[Cor.~3.1 (3)]{bsy}.
\end{proof}

For $y=-1$, Lemma \ref{lem:Gr} recovers
$c_{\mathrm{SM}}\left({\rm Gr}(r,E)\right) = c\left(S^\vee\otimes Q\right) \cap \pi^* c_{\mathrm{SM}}(X)$, treated in \cite[Proposition 1.5]{pp}.

\subsection{Fibrations}
\label{ss:Fibra}
Given a proper morphism $p\colon Y\rightarrow X$, let $\mathscr{X}:=\{X_k\}_{k\in K}$ be a stratification of $X$ into locally closed strata $X_k$ such that $p$ is locally trivial in the Zariski topology over each $X_k$ with smooth fiber~$F_k$. Assume that there exists a unique top-dimensional stratum $X_0$ in $\mathscr{X}$.
Let $p_!$ be the K-theoretic push-forward via $p$.

\begin{lemma}
\label{lemma:Fibra}
We have
\begin{align*}
p_! \, mC\left(Y\rightarrow X\right) &= \sum_{k\in K} d_k \, mC\left(\overline{X}_k \hookrightarrow X\right),\\
p_* \, T_y\left(Y\rightarrow X\right) &= \sum_{k\in K} e_k \, T_y\left(\overline{X}_k \hookrightarrow X\right)
\end{align*}
with
\begin{align*}
d_k &:= \left(\int_{F_k}\lambda_y\left(\mathscr{T}^\vee _{F_k}\right)\right)-{\sum_{j} d_j}, & e_k &:= \left(\int_{F_k}T_y\left(F_k\right)\right)-{\sum_{j} e_j}
\end{align*}
where the sums are over $j$ such that $ X_k\subset \overline{X}_j$.
\end{lemma}

\begin{proof}
Since $p$ is locally trivial in the Zariski topology over each $X_k$ with smooth fiber~$F_k$,
factoring $p^{-1}(X_k)\subset Y \xrightarrow{p} X$ as $p^{-1}(X_k)\xrightarrow{p} X_k \hookrightarrow X$ and by multiplicativity of $\lambda_y$, one has
\[
p_! \, mC\left(p^{-1}(X_k)\rightarrow X\right) =  \left(\int_{F_k}\lambda_y\left(\mathscr{T}^\vee _{F_k}\right)\right) \, mC\left({X}_k \hookrightarrow X\right),
\]
for each $k\in K$. Therefore, one has $d_0=\int_{F_0}\lambda_y\left(\mathscr{T}^\vee _{F_0}\right)$.
One verifies the formula for the coefficient $d_k$ by recursion on the codimension of the strata in $\mathscr{X}$.
The proof for $T_y$ is identical.
\end{proof}

In the case $y=-1$, Lemma \ref{ss:Fibra} was treated in \cite[Proposition 1.6]{pp}.

\subsection{Hirzebruch $\chi_y$-genus}
For a compact $X$, its \textit{Hirzebruch $\chi_y$-genus}  is
\[
\chi_y (X) := \int_X T_y(X) \qquad \in \mathbb{Q}[y].
\]
This invariant recovers the topological Euler characteristic for $y=-1$, the holomorphic Euler characteristic for $y=0$, and the signature for $y=1$ \cite{bsy}.
The Hirzebruch $\chi_y$-genus extends to arbitrary $X$ via the additivity in $K_0(\mbox{var}/\mathrm{pt})$.

\begin{example}
\label{ex:TyPnAn}
For the projective space $\mathbb{CP}^n$, one has
\[
\chi_y\left(\mathbb{CP}^n\right)=\int_{\mathbb{CP}^n} T_y\left(\mathbb{CP}^n\right)= 1 + (-y) + \cdots + (-y)^n.
\]
Consequently, one has
$\chi_y\left(\mathbb{A}^n_\mathbb{C}\right) = (-y)^n$.
\end{example}


\section{Motivic classes of a resolution of vexillary degeneracy loci}
\label{sec:motref}

Since motivic Chern and Hirzebruch classes have an equivalent formalism, as exemplified by Lemmata \ref{lem:zeros}, \ref{lem:Gr}, \ref{lemma:Fibra}, for simplicity we consider  only the case of motivic Hirzebruch classes  from this point on.
The main result of this section  is the following Theorem \ref{thm:TyOmegatilde},
computing  the motivic Hirzebruch class of a \textit{resolution} of  vexillary degeneracy loci. This is the explicit version of Theorem \ref{thm:secondintrothm}.

 \subsection{The locus \texorpdfstring{$\widetilde\Omega_{\bm{\tau}}$}{Omegatildetau}}
 \label{sec:Omegatildedef}
 Recall the geometric setup of \S\ref{sec:setup} defining the vexillary degeneracy locus $W_{\bm{\tau}}$ in a variety $X$ given maps of vector bundles \eqref{eq:maps} and a weakly increasing sequence $\bm{k}=(k_1, \dots, k_t)$. We define here a resolution $\widetilde\Omega_{\bm{\tau}}$ of $W_{\bm{\tau}}$.
  
As in \S\ref{sec:kpqbar}, we can reduce to the case when the triple $\bt=(\bm{k},\bp,\bq)$ is essential, that is, it corresponds to a minimal set of conditions $\dim \mathrm{ker}\left(E_{p_i}\rightarrow F_{q_i}\right)|_x\geq k_i$ for a point $x\in W_{\bt}$.
The triple $\bt$ could then be inflated to the triple $\bt'$ as in \S\ref{sec:tau'} and $W_{\bt}\cong W_{\bt'}$. 
 
Consider the variety $X_{k_t}$ parametrizing full flags of sub-bundles $V_{1} \subseteq \cdots \subseteq V_{k_t}$ with $\mathrm{rank}(V_{i})=i$ and $V_{i}\subseteq E_{p'_i}$ for each $i$. 
The variety $X_{k_t}$ is constructed as a sequence of projective bundles
\begin{equation*}
X=:X_0 \xleftarrow{\pi_1} \mathbb{P}\left(E_{p'_1}\right)=: X_1 \xleftarrow{\pi_2} \mathbb{P}\left(E_{p'_2}/\mathbb{S}_1\right)=: X_2 
\cdots  \xleftarrow{\pi_{k_t}} \mathbb{P}\left(E_{p'_{k_t}}/\mathbb{S}_{k_t-1}\right)=: X_{k_t},
\end{equation*}
where $\mathbb{S}_{i}/\mathbb{S}_{i-1}$ is the tautological line bundle on $X_i$, for each $i$. 
Here, we omit the obvious pull-backs via the natural projections $\pi_i$ to simplify the notation.
Since $X$ is assumed to be smooth, $X_i$ is also smooth, for each $i$. Let $\pi\colon X_{k_t}\rightarrow X$ be the natural projection.
Define
\[
\widetilde{\Omega}_{\bt} :=\left\{ \left(x, V_{1}\!\subseteq \!\cdots \!\subseteq\! V_{k_t}\right)\in X_{k_t} \, : \,  
V_{i}\subseteq \ker\left(E_{p'_i} \rightarrow F_{q'_i}\right)\Big|_x \,\mbox{ for all $i$}
\right\}
\]
with natural inclusion  $\iota\colon \widetilde{\Omega}_{\bt}\hookrightarrow X_{k_t}$. 
(We  study a quotient $\widetilde{\Omega}\rightarrow\Omega$ in \S\ref{sec:Omega}, hence the use of the tilde here.)
The restriction $\phi$ of $\pi$ to $\widetilde{\Omega}_{\bt}$ is a resolution of singularities as in Kempf-Laksov \cite{laksov1974determinantal}.
One has a commutative diagram
\[
\begin{tikzcd}
\widetilde{\Omega}_{\bt} \arrow{d}[swap]{\phi} \arrow[hookrightarrow]{r}{\iota} \arrow{d} & X_{k_t} \arrow{d}{\pi}\\
W_{\bt} \arrow[hookrightarrow]{r} & X
\end{tikzcd}
\]
and the fiber of $\phi$ over a  point $x$ in $W_{\bt}$ is
\[
\left\{ \left(V_{1}\!\subseteq \!\cdots \!\subseteq\! V_{k_t} \right)\in \pi^{-1}(x) \, :\, V_{i}\subseteq \ker\left(E_{p'_i} \rightarrow F_{q'_i}\right)\Big|_x \,\mbox{ for all $i$} \right\}.
\]
Note that since $(\bt')'=\bt'$, one has $\widetilde{\Omega}_{\bt}=\widetilde{\Omega}_{\bt'}$ by definition.

In general, the locus $\widetilde{\Omega}_{\bt}$ is not the \textit{minimal} resolution of $W_{\bt}$. The minimal resolution of $W_{\bt}$ is given by the variety parametrizing flags $V_{k_1}\!\subseteq \!\cdots \!\subseteq\! V_{k_t} $ such that $\mathrm{rank}(V_{k_i})=k_i$ and $V_{k_i}\subseteq \ker\left(E_{p_i} \rightarrow F_{q_i}\right)$ for each $i$. There is a forgetful map from $\widetilde{\Omega}_{\bt}$ to the minimal resolution of  $W_{\bt}$ which is birational. We find a closed formula precisely for the (push-forward of the) motivic Hirzebruch class of $\widetilde{\Omega}_{\bt}$.

Define $\bm{\lambda}=\bm{\lambda}_{\bt}$ as in \S \ref{sec:la}, and   bundles $E(i),F(i)$,  classes $c(i)$, and operators $R_i$, $T_y(R_i)$, and $T_y(i)$ as in \S\ref{sec:cidef}. 
Let $c(i)_j$ be the term of degree $j$ in $c(i)$.

\begin{theorem}
\label{thm:TyOmegatilde}
With assumptions as in \S\ref{sec:assumptions}, the  class $(\pi\iota)_* \, T_y\left(\widetilde{\Omega}_{\bt}\right)$ is 
\[
\frac{1}{ \prod_{(i,j)\in S}T_y\left(R_j - R_i\right)}\left|  \frac{1}{T_y\left(i\right)} \, c\left(i \right)_{\lambda_i +j-i}\right|_{1\leq i,j\leq k_t} \cap  T_y(X),
\]
where $S:=\{(i,j) : i\leq k_a< j, \mbox{ for some $a$}\}$ and $T_y\left(R_j - R_i\right):=Q_y\left(R_j - R_i\right)$.
Equivalently, this is
\[
 \frac{1}{\prod_{(i,j)\in S} T_y\left(R_j - R_i\right)} \,\prod_{i=1}^{k_t}  \frac{1}{T_y\left( i \right)} \, \iota_* \left[W_{\bt} \right]  \cap  T_y(X).
\]
\end{theorem}

The second expression in the statement follows from the first since 
\[
(\pi\iota)_*\left[ \widetilde{\Omega}_{\bt}\right] = \iota_*\left[W_{\bt} \right] = \left| c_{\lambda_i +j-i}\left(i \right)\right|_{1\leq i,j\leq k_t} \,\in\, A^*(X)
\]
by
 \cite{MR1154177}. For $y=-1$, the  class $(\pi\iota)_*\, c_{\textrm{SM}}\left(\widetilde{\Omega}_{\bt}\right)$ admits a simpler expression, see~\S\ref{CSMOmegatilde}.

\begin{remark}
We emphasize that all raising operators in  Theorem \ref{thm:TyOmegatilde} apply to the Chern classes in the expansion of $\left[W_{\bt} \right]$, and do not apply to the Chern classes contributed from the operators $T_y\left(i\right)$. 
\end{remark}

The first few terms in the expansion of the operator in Theorem \ref{thm:TyOmegatilde} are given  by multiplying
\[
 \frac{1}{T_y\left(R_j - R_i\right)} = 1 + \frac{1}{2}(y-1)\left(R_j - R_i\right) + \frac{1}{6}(y^2-4y+1)\left(R_j - R_i\right)^2 + \dots
\]
for $(i,j)\in S$, and 
\begin{eqnarray*}
\frac{1}{T_y\left(i \right)} &=& 1 + \frac{1}{2}(y-1) \left(\mathrm{ch}_0(i)R_i+ c_1(i)\right) \\
&&{}- \frac{1}{12}(1+y)^2 \left( \mathrm{ch}_0(i) R_i^2 + 2 c_1(i)R_i + \mathrm{ch}_2(i) \right)\\
&&{}+\frac{1}{4}(1-y)^2 \bigg( {\mathrm{ch}_0(i)+1\choose 2}R_i^2 + (\mathrm{ch}_0(i)+1)c_1(i)R_i \\
&&{}\qquad\qquad\qquad\qquad\qquad + c_2(E(i)-F(i))\bigg) +\dots
\end{eqnarray*}
where $\mathrm{ch}(i):=\mathrm{ch}(F(i)-E(i))$, for each $i$.

As an example for the set $S$, for $\bm{k}=(2,5)$, one has \
\[
S=\{(1,3),(2,3),(1,4),(2,4),(1,5),(2,5)\}.
\]

\begin{remark}
\label{rmk:fullflags}
The definition of the locus $\widetilde{\Omega}_{\bt}$ requires  flags of vector bundles finer than the ones given in \eqref{eq:maps} precisely when $k_i>i$ for some $i$.
However, the expression in the statement only depends on the two given  flags in \eqref{eq:maps}, and not on the choice of finer flags 
$E_{p'_1}\subseteq  \cdots \subseteq E_{p'_{k_t}}$ and $F_{q'_1}\twoheadrightarrow \cdots \twoheadrightarrow   F_{q'_{k_t}}$ used to define the locus $\widetilde{\Omega}_{\bt}$.

In fact, the flags  in \eqref{eq:maps} can be extended to full flags of vector bundles after passing to an appropriate projective bundle over $X$.
Namely, suppose that $p_i>p_{i-1}$. To define a vector bundle $E_{p_i-1}$ of rank $p_i-1$ such that $E_{p_{i-1}}\subset E_{p_i -1} \subset E_{p_i}$, consider the projective bundle
\begin{equation}
\label{eq:X'step1}
 \mathbb{P}\left(E_{p_i}/E_{p_{i-1}}\right) \rightarrow{} X
\end{equation}
and set $E_{p_i-1}:=Q$, where $Q/E_{p_{i-1}}$ is the universal quotient bundle on  $\mathbb{P}\left(E_{p_i}/E_{p_{i-1}}\right)$ --- here again we omit to denote when a bundle is pulled back for simplicity. Proceeding in this way, one constructs full flags $E_{1}\subseteq E_2\subseteq \cdots \subseteq E_{p_t}$ and $F_{q_1}\twoheadrightarrow \cdots \twoheadrightarrow F_2\twoheadrightarrow   F_{1}$ defined on a space $X'$ obtained as a tower of projective bundles  $p\colon X'\rightarrow X$.

The theorem then applies to give the class 
\begin{equation}
\label{eq:Omegatildetau'}
\frac{1}{ \prod_{(i,j)\in S}T_y\left(R_j - R_i\right)}\left|  \frac{1}{T_y\left(i\right)} \, c\left(i \right)_{\lambda_i +j-i}\right|_{1\leq i,j\leq k_t} \cap  T_y\left(X'\right) \quad\in A_*\left(X'\right)[y].
\end{equation}
Factoring $p$ as a composition of projective bundles, $A_*(X)$ can be identified as a subring of $A_*\left(X'\right)$ via $p^*$.
After Lemma \ref{Gr}, the class $\left(\pi\iota\right)_* \,T_y\left(\widetilde{\Omega}_{\bt} \right)$ in $p^* \left( A_*(X)[y]\right)$ is recovered after quotienting \eqref{eq:Omegatildetau'} by the motivic Hirzebruch class of the  consecutive fibers (the fibers of $p$ are locally constant). Indeed, this has the effect of replacing $T_y\left(X'\right)$ with $p^*\, T_y\left(X\right)$ in \eqref{eq:Omegatildetau'}, thus it recovers the formula in the theorem.
\end{remark}

\smallskip

To prove the theorem, we distinguish  four cases, following the proof strategy from \cite{af1}. In the basic and dominant cases (\S\S\ref{sec:basic}--\ref{sec:dominant}), we immediately compute the motivic Hirzebruch class of the corresponding degeneracy loci. In the main and general cases (\S\S\ref{sec:maincase}--\ref{sec:generalcase}), where the degeneracy loci are singular, we  compute here the motivic Hirzebruch class of the \textit{push-forward of the  resolution} $\widetilde\Omega$ of the degeneracy loci. In the last two cases, the computation of the motivic Hirzebruch class of the actual degeneracy loci is completed as in Theorem \ref{thm:final} building on Theorem \ref{thm:TyOmegatilde}.

\subsection{Basic case} 
\label{sec:basic}
Let $L$ be a line bundle and $E$ a vector bundle of rank $e$ on~$X$. 

\begin{lemma}
\label{lem:basic}
If the locus $\iota \colon Z\hookrightarrow X$ where a map $\varphi \colon L \rightarrow E$ vanishes is smooth of codimension $e$, its motivic Hirzebruch class is
\[
\iota_* \,T_y(Z)= \frac{1}{T_y\left(E\otimes R \right)}\,c_e\left(E- L\right)  \cap  T_y\left( X \right),
\]
where $R$ is the raising operator acting on $c\left(E- L\right)$.
\end{lemma}

\begin{proof}
From Lemma \ref{lem:zeros}, we have
\[
\iota_* \,T_y(Z)= \frac{c_e\left(E\otimes L^\vee\right)}{T_y\left(E\otimes L^\vee \right)}  \cap  T_y\left( X \right).
\]
The statement follows from the identities $c_e\left(E\otimes L^\vee\right)=c_e\left(E- L\right)$ and
\begin{equation}
\label{eq:firstraising}
\left( c_1\left(L^\vee\right) \right)^k c_i\left(E-L\right) = c_{i+k}\left(E-L\right), \qquad \mbox{for $k\geq 0$ and $i\geq e$}
\end{equation}
(see e.g., \cite[pg.~3]{af1}), so that multiplication by $c_1\left(L^\vee\right)$ here coincides with the operator~$R$.
\end{proof}

\subsection{Dominant case} 
\label{sec:dominant}
Consider maps of vector bundles on $X$
\[
E_1\hookrightarrow \dots \hookrightarrow E_t \rightarrow F_{q_1} \twoheadrightarrow \dots \twoheadrightarrow F_{q_t}
\]
with ${\rm rank}(E_i)=i$ and ${\rm rank}(F_{q_i})=q_i$. Here, we assume that $\bt=(\bm{k},\bp,\bq)$ is a triple with $k_i = p_i =i$, for $1\leq i \leq t$.
Consider the degeneracy locus 
\[
W_{\bt} :=\left\{ x\in X \, : \, \left(E_i\rightarrow F_{q_i}\right)|_x \mbox{ is zero for all $i$} \right\}
\]
with inclusion $\iota\colon W_{\bt} \hookrightarrow X$. In this case, $k_i=i$ and $\lambda_i= q_i$, for each~$i$. 

\begin{lemma}
When $W_{\bt}$ is smooth of dimension $\dim X - \bm{\lambda}_{\bt}$,
the class $\iota_* \,T_y (W_{\bt})$ is given by
\begin{align*}
\left| \frac{1}{T_y\left(R_i\otimes F_{q_i}\right)}\, c_{\lambda_i+j-i}\left(F_{q_i}-E_i\right) \right|_{1\leq i,j\leq t}
\cap  T_y\left( X \right).
\end{align*}
\end{lemma}

\begin{proof}
From Lemma \ref{lem:basic}, it follows that
\[
\iota_* \,T_y \left(W_{\bt}\right) = \prod_{i=1}^t \frac{1}{T_y\left(F_{q_i} \otimes \left(E_i/E_{i-1}\right)^\vee\right)}\, c_{\lambda_i}\left(F_{q_i}-E_i/E_{i-1}\right) \cap  T_y\left( X \right).
\]
As in \eqref{eq:firstraising}, one has that $c_1((E_{i}/E_{i-1})^\vee)$ acts as $R_i$ on $c_{\lambda_i}(F_{q_i}-E_i/E_{i-1})$. 
By means of the identity
\begin{equation}
\label{eq:detdominant}
\prod_{i=1}^t  c_{\lambda_i}\left(F_{q_i}-E_i/E_{i-1}\right) 
=
\left| c_{\lambda_i+j-i}\left(F_{q_i}-E_i\right) \right|_{1\leq i,j\leq t},
\end{equation}
valid as in \cite[\S 1.2]{af1}, 
the statement follows.
\end{proof}

\subsection{Main case} 
\label{sec:maincase}
Consider maps of vector bundles on $X$
\[
E_{p_1}\hookrightarrow \dots \hookrightarrow E_{p_t} \rightarrow F_{q_1} \twoheadrightarrow \dots \twoheadrightarrow F_{q_t}
\]
with ${\rm rank}(E_{p_i})=p_i$ and ${\rm rank}(F_{q_i})=q_i$. Here, we assume that $\bt=(\bm{k},\bp,\bq)$ is a triple where $k_i=i$, for $1\leq i \leq t$, and $q_1 -p_1+1>\cdots> q_t -p_t+t$. 
In particular, the triple $\bt$ is essential and $\bt=\bt'$ (notation as in \S\ref{sec:setup}). 
We would like to compute the motivic Hirzebruch class of the locus 
\[
W_{\bt} :=\left\{ x\in X \, : \, \dim {\rm ker}\left(E_{p_i}\rightarrow F_{q_i}\right)|_x \geq i \right\}.
\]
Instead, we  compute here the motivic Hirzebruch class of the \textit{resolution} of $W_{\bt}$; this will be  used in \S\ref{CSMW} to  obtain the motivic Hirzebruch  class of $W_{\bt}$.

Consider the following sequence of projective bundles
\begin{equation*}
X=:X_0 \xleftarrow{\pi_1} \mathbb{P}\left(E_{p_1}\right)=: X_1 \xleftarrow{\pi_2} \mathbb{P}\left(E_{p_2}/\mathbb{S}_1\right)=: X_2 
\cdots  \xleftarrow{\pi_t} \mathbb{P}\left(E_{p_t}/\mathbb{S}_{t-1}\right)=: X_t,
\end{equation*}
where $\mathbb{S}_{i}/\mathbb{S}_{i-1}$ is the tautological line bundle on $X_i$.
The variety $X_t$ here parametrizes flags of sub-bundles $V_1\subset \cdots \subset V_t$ with $\mathrm{rank}\left( V_i\right)=i$ and $V_i\subseteq E_{p_i}$.
 Since $X$ is assumed to be smooth, $X_i$ is also smooth, for each $i$.
Define
\[
\widetilde{\Omega}_{\bt} :=\left\{ \left(x, V_1\subset \cdots \subset V_t \right)\in X_t \, : \,  
V_i\subseteq \mathrm{ker}\left( E_{p_i} \rightarrow F_{q_i}\right)|_x \mbox{ for all $i$}
\right\}
\]
with natural maps  $\iota\colon \widetilde{\Omega}_{\bt}\hookrightarrow X_t$ and $\pi\colon X_t\rightarrow X$. 

\begin{proposition}
\label{prop:Hirzclassmain}
With assumptions as in \S\ref{sec:assumptions}, the  class $(\pi\iota)_* \, T_y\left(\widetilde{\Omega}_{\bt}\right)$ is 
\[
\frac{1}{\prod_{i<j} T_y\left(R_j - R_i\right)}\left|  \frac{1}{T_y\left(R_i\otimes \left( F_{q_i} -E_{p_i}\right)\right)}\, c_{\lambda_i +j-i}\left(F_{q_i} - E_{p_i} \right)\right|_{1\leq i,j\leq t} \cap  T_y(X).
\]
Equivalently, this is
\[
 \frac{1}{\prod_{i<j} T_y\left(R_j - R_i\right)}\prod_{i=1}^{t} \frac{1}{T_y\left(R_i\otimes \left(F_{q_i}-E_{p_i} \right)\right)} \left[W_{\bt} \right]  \cap  T_y(X).
\]
\end{proposition}

Note that for $\bm{k}=(1,\dots,t)$, one has $E(i)=E_{p_i}$, $F(i)=F_{q_i}$, and $S=\{(i,j) : i<j\}$, hence the  formula in Theorem \ref{thm:TyOmegatilde} specializes to the one in Proposition \ref{prop:Hirzclassmain}.

\begin{proof}
From Lemma \ref{lem:Gr}, we have
\begin{align*}
\label{CSMXi}
T_y(X_i) 
=T_y \left((\mathbb{S}_i/\mathbb{S}_{i-1})^\vee \otimes E_{p_i}/\mathbb{S}_i \right) \cap \pi_i^* \,T_y \left(X_{i-1}\right).
\end{align*}
Combining this with the basic case, one has
\[
\iota_* \, T_y\left(\widetilde{\Omega}_{\bt}\right)
=\prod_{i=1}^t \frac{T_y\left(\left(\mathbb{S}_i/\mathbb{S}_{i-1}\right)^\vee\otimes E_{p_i}/\mathbb{S}_i\right)}{T_y\left(\left(\mathbb{S}_i/\mathbb{S}_{i-1}\right)^\vee\otimes F_{q_i} \right)}\, c_{q_i}\left(F_{q_i} - \mathbb{S}_i/\mathbb{S}_{i-1} \right) \cap \pi^* T_y(X).
\]
As in \eqref{eq:firstraising}, one has that $c_1((\mathbb{S}_{i}/\mathbb{S}_{i-1})^\vee)$ acts as $R_i$ on $c_{q_i}(F_{q_i}-\mathbb{S}_i/\mathbb{S}_{i-1})$. 
The statement follows since one has
\begin{align*}
\begin{split}
T_y\left(R_i \otimes \mathbb{S}_i\right) &= T_y\left(R_i - c_1\left(\mathbb{S}_1^\vee\right)\right) \cdots T_y\left(R_i - c_1\left(\left(\mathbb{S}_{i}/\mathbb{S}_{i-1}\right)^\vee\right)\right)\\
&= T_y\left(R_i-R_1\right) \cdots T_y\left(R_i-R_{i-1}\right),
\end{split}
\end{align*}
and
\[
\prod_{i=1}^t  c_{q_i}\left(F_{q_i}-\mathbb{S}_i/\mathbb{S}_{i-1}\right) 
=
\left| c_{q_i+j-i}\left(F_{q_i}-\mathbb{S}_i\right) \right|_{1\leq i,j\leq t} \nonumber
\]
as in  \eqref{eq:detdominant}, 
and
\[
\pi_* \left| c_{q_i+j-i}\left(F_{q_i}-\mathbb{S}_i\right) \right|_{1\leq i,j\leq t}
= \left| c_{\lambda_i+j-i}\left(F_{q_i}-E_{p_i}\right) \right|_{1\leq i,j\leq t}
\]
as in \cite[\S 1.3]{af1}.
\end{proof}

\subsection{General case}
\label{sec:generalcase}
We consider here the general case of Theorem \ref{thm:TyOmegatilde} and thus complete its proof.

\begin{proof}[Proof of Theorem \ref{thm:TyOmegatilde}]
After \S\ref{sec:kpqbar}, we can restrict to the case of essential triples $\bt=(\bm{k},\bp,\bq)$. 
Indeed, the formula in the statement remains invariant after replacing a triple $\bt=(\bm{k},\bp,\bq)$ with $\bar{\bt}=\left(\bar{\bm{k}},\bar{\bp},\bar{\bq}\right)$ such that $\bar{\bt}$ is essential and  $W_{\bt}= W_{\bar{\bt}}$, as in  \S\ref{sec:kpqbar}, and thus $\widetilde\Omega_{\bt}\cong \widetilde\Omega_{\bar{\bt}}$.

The case when $\bm{k}=(1,2,\dots,t)$ is treated by Proposition \ref{prop:Hirzclassmain}. Otherwise, to define $\widetilde\Omega_{\bt}$, the triple $\bt$ is inflated to a triple $\bt'=(\bm{k}',\bp',\bq')$ such that the sequences $\bm{k}', \bm{p}'$, and $\bm{q}'$ are of length $k_t$ with $\bm{k}'=(1,2,\dots, k_t)$. 
One has $\bm{\lambda}_{\bt}=\bm{\lambda}_{\bt'}$ and $\widetilde{\Omega}_{\bt}=\widetilde{\Omega}_{\bt'}$ by definition. 
We can thus apply Proposition \ref{prop:Hirzclassmain} to compute $(\pi\iota)_* \,T_y\left(\widetilde{\Omega}_{\bt'}\right)$ with respect to the flag $E_{p'_1}\subseteq  \cdots\subseteq E_{p'_{k_t}}$ refining $E_{p_1}\subseteq  \cdots\subseteq E_{p_t}$, and the flag $F_{q'_1}\twoheadrightarrow\cdots \twoheadrightarrow F_{q'_{k_t}}$ refining $F_{q_1}\twoheadrightarrow\cdots \twoheadrightarrow F_{q_t}$. After Remark \ref{rmk:fullflags}, we can assume that such finer flags exist on $X$.
Thus from Proposition \ref{prop:Hirzclassmain}, the class $(\pi\iota)_* \,T_y\left(\widetilde{\Omega}_{\bt'}\right)$ is
\[
\frac{1}{\prod_{i<j} T_y\left(R_j - R_i\right)}\left|  \frac{1}{T_y\left(R_i\otimes \left( F_{q'_i} -E_{p'_i}\right)\right)}\, c_{\lambda_i +j-i}\left(F_{q'_i} - E_{p'_i} \right)\right|_{1\leq i,j\leq k_t} \cap  T_y(X).
\]
It remains to verify that this formula is equivalent to the formula in the statement.

As in \cite[\S 1.4]{af1} (see also \cite[\S 1]{anderson2019k}), one has
\begin{equation}
\label{eq:detdet}
\left| c(i)_{\lambda_i+j-i} \right|_{1\leq i,j\leq k_t}= \left| c_{\lambda_i+j-i}(F_{q'_i}-E_{p'_i}) \right|_{1\leq i,j\leq k_t}.
\end{equation}
Indeed, from the definition of $c(i)$ (see \S\ref{sec:cidef}), the entries of the two determinants in \eqref{eq:detdet} do not match, but the determinants do.
Furthermore, fix $i$ with $1\leq i \leq k_t$. One has $k_{a-1}< i \leq k_a$ for some $a$. We claim that  
\begin{equation}
\label{eq:claimTyi}
\prod_{j: j>k_a} \frac{1}{T_y\left(R_j-R_i\right)}\cdot\frac{1}{T_y(i)}  
\quad=\quad  \prod_{j: j>i} \frac{1}{T_y\left(R_j-R_i\right)}\cdot\frac{1}{T_y\left(R_i\otimes \left(F_{q'_i}-E_{p'_i}\right)\right)} 
\end{equation}
as operators on \eqref{eq:detdet}. Recall that $T_y(i):=T_y\left( R_i \otimes \left( F(i)-E(i)\right)\right)$, where $E(i)=E_{p_a}$ and $F(i)=F_{q_a}$, by definition (\S\ref{sec:cidef}). Assume that $q'_i=q_a$. Then one has
\[
q'_i=q'_{i+1}=\cdots =q'_{k_a}=q_a \quad \mbox{and} \quad \left(p'_i, p'_{i+1}, \dots, p'_{k_a}\right)=\left(p_a-k_a+i,\dots,p_a-1,p_a\right)
\]
by construction. It follows that
\[
\frac{1}{T_y\left( i \right)}=
\frac{1}{T_y\left(R_i\otimes \left(F_{q'_i}-E_{p'_{k_a}}\right) \right)} = 
\frac{T_y\left( R_i\otimes \left(E_{p'_{k_a}}/ E_{p'_i}\right)\right)}{T_y\left(R_i\otimes \left(F_{q'_i} -E_{p'_i}\right)\right)} .
\]
The second equality follows from the multiplicativity of $T_y$ on exact sequences.
Furthermore, one has
\[
T_y\left(E_{p'_{k_a}}/ E_{p'_i}\right)=\prod_{j:i<j\leq k_a} T_y\left(c_1\left(E_{p'_j}/ E_{p'_{j-1}}\right)\right) =\prod_{j:i<j\leq k_a} T_y\left(-R_j\right)
\]
as operators on \eqref{eq:detdet}. The second equality follows from \eqref{eq:firstraising}: 
for $j$ such that $i<j\leq k_a$, the indices of the Chern class $c\left(F_{q'_j}-E_{p'_{j}}\right)=c\left(F_{q'_j}-E_{p'_{j-1}}-E_{p'_{j}}/E_{p'_{j-1}}\right)$ in the $j$-th row of the right-hand side of \eqref{eq:detdet} are at least 
\[
\lambda_j+1-j=q'_j - p'_j +1= q'_j - p'_{j-1} =\mathrm{rank}\left(F_{q'_j}-E_{p'_{j-1}}\right), 
\]
hence from \eqref{eq:firstraising},
$c_1\left(E_{p'_j}/ E_{p'_{j-1}}\right)$ acts as the operator $-R_j$ on \eqref{eq:detdet}.
Then, 
one has
\begin{equation*}
T_y\left( R_i\otimes\left( E_{p'_{k_a}}/ E_{p'_i}\right)\right) = \prod_{j:i<j\leq k_a} T_y\left( R_i\otimes c_1\left(E_{p'_j}/ E_{p'_{j-1}} \right)\right) 
=\prod_{j:i<j\leq k_a} T_y\left( R_i-R_j\right)
\end{equation*}
as operators on \eqref{eq:detdet}. 
It follows that 
\[
\frac{1}{T_y\left( i \right)}=
\prod_{j:i<j\leq k_a} \frac{1}{T_y\left( R_j - R_i\right)}
\cdot 
\frac{1}{T_y\left(R_i\otimes \left(F_{q'_i} -E_{p'_i}\right)\right)} 
\]
as operators on \eqref{eq:detdet}, whence the claim \eqref{eq:claimTyi}.
When $q'_i<q_a$, then necessarily $p_a =p'_i$, and this case is treated similarly.
We thus conclude that the formula for $(\pi\iota)_* \, T_y\left(\widetilde{\Omega}_{\bt'}\right)$ from Proposition \ref{prop:Hirzclassmain} is equivalent to the formula in the statement.
\end{proof}


\section{CSM classes of a resolution of vexillary degeneracy loci}
\label{CSMOmegatilde}

In the case of the CSM class, the results of the previous section simplify as follows. 
For a triple $\bt$, define $\bm{\lambda}=\bm{\lambda}_{\bt}$ as in \S \ref{sec:la}, and   bundles $E(i),F(i)$,  classes $c(i)$, and operators $R_i$ and $T_y(i)$ as in \S\ref{sec:cidef}. 
Let $c(i)_j$ be the term of degree $j$ in $c(i)$, and given a variable $t$, define $c_t(i):=\sum_{j\geq 0} c(i)_j \,t^j$.  We use below the virtual rank  $\mathrm{ch}(i)_0=\mathrm{rank}(F(i)-E(i))$.

\begin{theorem}
\label{thm:csmOmegatilde}
With assumptions as in \S\ref{sec:assumptions}, the class $(\pi\iota)_* \,c_{\rm SM} \left(\widetilde{\Omega}_{\bt}\right)$ is given by
\begin{align*}
\prod_{(i, j)\in S}  \frac{1}{1+R_j-R_i} 
\left| \frac{\left( 1+R_i\right)^{-\mathrm{ch}(i)_0}}{c_{\frac{1}{1+R_i}}(i)} c(i)_{\lambda_i+j-i} \right|_{1\leq i,j\leq k_t} 
\cap  c_{\rm SM}\left( X \right),
\end{align*}
where $S:=\{(i,j) : i\leq k_a< j \mbox{ for some $a$}\}$.
Equivalently, this is
\[
\prod_{(i, j)\in S}  \frac{1}{1+R_j-R_i} \,\prod_{i=1}^t  \frac{\left( 1+R_i\right)^{-\mathrm{ch}(i)_0}}{c_{\frac{1}{1+R_i}}(i)}   \left[W_{\bt} \right]  \cap  c_{\rm SM}\left( X \right).
\]
\end{theorem}

We emphasize that in the above formula all raising operators apply to the Chern classes coming from the expressions of type $c(i)_{\lambda_i+j-i}$ in the expansion of $\left[W_{\bt} \right]$, and do not apply to the Chern classes contributed from the terms $c_{\frac{1}{1+R_i}}(i)$.

 The first few terms in the expansion of the operator in Theorem \ref{thm:csmOmegatilde}  are given  by multiplying
\[
 \frac{1}{1+R_j - R_i} = 1 -(R_j - R_i) + (R_j - R_i)^2 + \dots
\]
for $(i,j)\in S$, and 
\begin{eqnarray*}
\frac{1}{c_{\frac{1}{1+R_i}}(i)} &=& 1 - \frac{c(i)_1}{1+R_i}  + \left(c(i)^2_1 -c(i)_2 \right) \left(\frac{1}{1+R_i} \right)^2 +\dots\\
&=& 1 -c(i)_1 + c(i)_1R_i +c(i)^2_1 -c(i)_2 + \cdots
\end{eqnarray*}
 for each $i$, and 
 \[
 (1+R_i)^{-\mathrm{ch}(i)_0}=1-\mathrm{ch}(i)_0R_i + \frac{-\mathrm{ch}(i)_0 (-\mathrm{ch}(i)_0-1)}{2}R_i^2+\cdots 
 \]
 for each $i$.

To show Theorem \ref{thm:csmOmegatilde}, we use  the following technical lemma:

\begin{lemma}
\label{cvir}
Let $L$ be a line bundle and $E$ a virtual vector bundle of virtual rank $e$. One has
\[
c(E\otimes L) = (1+c_1(L))^e c_{\frac{1}{1+c_1(L)}}(E).
\]
\end{lemma}

\noindent Here, $c_t(E):=\sum_{i\geq 0} c_i(E)\,t^i$.
The case when $E$ is a vector bundle is \cite[Example 3.2.2]{MR1644323}. For a virtual vector bundle, the argument is similar.

\begin{proof}[Proof of Theorem \ref{thm:csmOmegatilde}]
Specializing Theorem \ref{thm:TyOmegatilde} at $y=-1$, we have that the class $(\pi\iota)_* \,c_{\rm SM} \left(\widetilde{\Omega}_{\bt}\right)$ is given by
\begin{align*}
\prod_{(i, j)\in S}  \frac{1}{1+R_j-R_i} 
\left| \frac{1}{c\left(R_i\otimes \left(F(i)-E(i) \right)\right)} c(i)_{\lambda_i+j-i} \right|_{1\leq i,j\leq k_t} 
\cap  c_{\rm SM}\left( X \right),
\end{align*}
where $c\left(R_i\otimes \left(F(i)-E(i) \right)\right):= T_{-1}(i)$ is the specialization at $y=-1$ of the operator $T_y(i)$ from \S\ref{sec:cidef}.
The virtual bundle $F(i)-E(i)$ has virtual rank equal to $\mathrm{ch}(i)_0=\mathrm{rank}\left(F(i)\right)-\mathrm{rank}\left(E(i)\right)$.
Applying Lemma \ref{cvir}, we have
\[
c\left(R_i\otimes \left(F(i)-E(i) \right)\right) = \left( 1+R_i\right)^{\mathrm{ch}(i)_0} c_{\frac{1}{1+R_i}}(i),
\]
hence the statement.
\end{proof}


\section{Motivic classes of vexillary degeneracy loci}
\label{CSMW}

For a triple $\bt$, recall the inflated triple $\bt'$ from \S\ref{sec:tau'}.
The aim of this section is to prove Theorem \ref{thm:firststep}, here  restated:

\begin{theorem}
\label{thm:1Omegatildefrom1W}
For a triple $\bt=(\bm{k}, \bm{p}, \bm{q})$ and with assumptions as in \S\ref{sec:assumptions}, one has
\[
(\pi\iota)_* \, T_y\left(\widetilde{\Omega}_{\bt}\right) = \sum_{\bm{k^+}} (-y)^{|\bm{k^+}|-|\bm{k}'|} \, \iota_* \, T_y\left(W_{\btp}\right)
\]
where $\btp=\left(\bm{k^+}, \bm{p}', \bm{q}'\right)$ and the sum is over the set of weakly increasing sequences $\bm{k^+}\geq\bm{k}'=(1,\dots, k_t)$. 
\end{theorem}

\subsection{The stratification}
\label{sec:filtr}
The degeneracy locus $W_{\bt}$ is stratified by the loci $W_{\btp}^\circ \subseteq W_{\bt}$ defined in \eqref{eq:locclosedstrata} with $\btp=\left(\bm{k^+}, \bm{p}', \bm{q}'\right)$,
for weakly increasing sequences $\bm{k^+}\geq\bm{k}'$. After \S\ref{sec:lafeasible}, we set $W_{\btp}$  empty, unless $\bm{\lambda}_{\btp}$ is weakly decreasing.
The codimension of $W_{\btp}$ in $W_{\bt}$ is $|\bm{\lambda}_{\btp}|-|\bm{\lambda}_{\bt}|$. There is no such stratum of codimension one in $W_{\bt}$.

\smallskip

Recall the variety $\widetilde{\Omega}_{\bt}$ from \S\ref{sec:motref}, with the following diagram
\[
\begin{tikzcd}
\widetilde{\Omega}_{\bt}\arrow{d}[swap]{\phi} \arrow[hookrightarrow]{r}{\iota} \arrow{d} & X_{k_t} \arrow{d}{\pi}\\
W_{\bt} \arrow[hookrightarrow]{r} & X.
\end{tikzcd}
\]
The fiber of $\phi$ over a  point $x$ in $W_{\bt}$ is
\[
\left\{ (V_{1}\!\subset \!\cdots \!\subset\! V_{k_t} )\in \pi^{-1}(x) \, :\, V_{i}\subseteq \ker\left(E_{p'_i} \rightarrow F_{q'_i}\right)\Big|_x \, \mbox{ for all $i$} \right\}.
\]
The map $\phi$ is locally trivial on each locally closed stratum $W_{\bt^+}^\circ\subseteq W_{\bt}$.

\begin{proof}[Proof of Theorem \ref{thm:1Omegatildefrom1W}]
The collection of loci $W_{\btp}^\circ$ from \eqref{eq:locclosedstrata}, with $\btp=\left(\bm{k^+}, \bm{p}', \bm{q}'\right)$
for $\bm{k^+}\geq\bm{k}'$,
 gives a stratification of the locus $W_{\bt}$ such that $\phi$ is locally trivial on each  $W_{\bt^+}^\circ$. 
Given $\bm{k^+}\geq\bm{k}'$, the generic fiber of $\phi$ on  $W_{\btp}$ is isomorphic to the Schubert variety associated to a partition $\bm{\nu}^+$ assigned to $\bm{k}^+$:
\[
S_{\bm{\nu}^+}:= \left\{ (V_{1}\!\subseteq \!\cdots \!\subseteq\! V_{k_t} )\in \textrm{Fl}\left(1,\dots,k_t;\mathbb{C}^{p_t}\right) \, :\, V_{i}\subseteq K_{k^+_i} \, \mbox{ for all $i$} \right\},
\]
where $K_0\subset\! \cdots\! \subset K_{p_t}$ is a fixed complete flag of vector spaces inside $\mathbb{C}^{p_t}$ with \mbox{$\dim(K_i)=i$.}
The vector spaces $K_{k^+_i}$ are meant to be identified with \mbox{$\ker(E_{p'_i} \rightarrow F_{q'_i})\big|_x$} for a generic point $x$ in $W_{\bt^+}$.
If $\bm{k^+}$ is strictly increasing, the partition  $\bm{\nu}^+$ is defined as
\[
\bm{\nu}^+ :=(k^+_{k_t}-k_t, \dots, k^+_1-1).
\]
In general, for $\bm{k^+}$  not necessarily strictly increasing, a coordinate computation shows that the Schubert cell $S_{\bm{\nu}^+}^\circ$ in $S_{\bm{\nu}^+}$ is isomorphic to the affine space $\mathbb{A}^{|\bm{k^+}|-|\bm{k}'|}$.

As in Lemma  \ref{lemma:Fibra}, one has 
\[
(\pi\iota)_* \, T_y\left(\widetilde{\Omega}_{\bt}\right) = \sum_{\bm{k^+}} d_{(\bm{k^+})} \,\iota_*\, T_y\left(W_{\btp}\right)
\]
with $d_{(\bm{k^+})}=\chi_y\left(S_{\bm{\nu}^+}\right)-\sum_{\bm{f}} d_{\bm{f}}$, where the sum is over the set of weakly increasing sequences $\bm{f}$ such that $W_{\bm{\phi}}\supset W_{\btp}$ for $\bm{\phi}=(\bm{f}, \bm{p}', \bm{q}')$, that is, $\bm{f}<\bm{k^+}$ and $\bm{\lambda}_{\bm{\phi}}$ is weakly decreasing.
Clearly $d_{(\bm{k}')}=1$, and by recursion one finds that $d_{(\bm{k^+})}$ is equal to the Hirzebruch $\chi_y$-genus of the {Schubert cell} $\left(S_{\bm{\nu}^+}\right)^\circ$.
Since $\left(S_{\bm{\nu}^+}\right)^\circ \cong \mathbb{A}^{|\bm{k^+}|-|\bm{k}'|}$, one has $d_{(\bm{k^+})}=(-y)^{|\bm{k^+}|-|\bm{k}'|}$ for each $\bm{k^+}$ (as in Example \ref{ex:TyPnAn}), hence the statement.
\end{proof}

\begin{example}
Consider the triple $\bt=(\bm{k},\bp,\bq)$ where $\bm{p}=(2,3)$, $\bm{q}=(3,2)$, and $\bm{k}=(1,2)$. One has
\[
(\pi\iota)_* \, T_y\left( \widetilde{\Omega}_{\bt}\right)  = 
\iota_*\,T_y\left({W_{(1,2)}}\right) 
-y \,\iota_*\,T_y\left({W_{(2,2)}}\right) 
-y \, \iota_*\,T_y\left({W_{(1,3)}}\right) 
+ y^2 \, \iota_*\,T_y\left({W_{(2,3)}}\right).
\]
Here for simplicity, we use the notation $W_{\bm{k^+}}:=W_{\btp}$.
For each stratum, one has the following configuration over its general point:
\begin{align*}
W_{(1,2)} :& \qquad \dim \textrm{Ker} (E_2\rightarrow F_3) =1 \quad \mbox{and} \quad \dim \textrm{Ker} (E_3\rightarrow F_2) =2\\
W_{(2,2)} :& \qquad \dim \textrm{Ker} (E_2\rightarrow F_3) =2 \quad \mbox{and} \quad \dim \textrm{Ker} (E_3\rightarrow F_2) =2\\
W_{(1,3)} :& \qquad\dim \textrm{Ker} (E_2\rightarrow F_3) =1 \quad \mbox{and} \quad \dim \textrm{Ker} (E_3\rightarrow F_2) =3\\
W_{(2,3)} :& \qquad\dim \textrm{Ker} (E_2\rightarrow F_3) =2 \quad \mbox{and} \quad \dim \textrm{Ker} (E_3\rightarrow F_2) =3.
\end{align*}
\end{example}

\subsection{Example: Schubert varieties in Grassmannians}
\label{sec:SchubinGr}
The Schubert varieties of a Grassmannian are \textit{Grassmannian degeneracy loci} for maps from the tautological vector bundle to a flag of constant bundles.
Our results apply to give the motivic Hirzebruch class of the Schubert varieties in terms of the motivic Hirzebruch class of the Grassmannian.
For instance, for $X=G_2\left(\mathbb{C}^5\right)$, consider the Schubert variety $S_{\bm{\lambda}}$ associated to the partition $\bm{\lambda}=(2,1)$. 
This is the degeneracy locus $W_{\bt}$ corresponding to the triple $\bt=\left(\bm{k}, \bm{p}, \bm{q} \right)$ with $\bm{k}=(1,2)$, $\bm{p}=(2,2)$, and $\bm{q}=(3,1)$.
The stratification here consists of two strata: the stratum with $\bm{k^+}=(1,2)$ and the one with $\bm{k^+}=(2,2)$.
For $y=-1$,  Theorems \ref{thm:1Omegatildefrom1W} and \ref{thm:TyOmegatilde} give
\begin{eqnarray*}
\iota_*  \,c_{{\rm SM}} \left(S_{\bm{\lambda}}\right) &= &
\left(
\begin{ytableau}
\hfil & \hfil \\
\hfil 
\end{ytableau}
-2\,
\begin{ytableau}
\hfil & \hfil & \hfil\\
\hfil 
\end{ytableau}
-2 \,
\begin{ytableau}
\hfil & \hfil \\
\hfil & \hfil 
\end{ytableau}
+5 \,
\begin{ytableau}
\hfil & \hfil & \hfil \\
\hfil & \hfil 
\end{ytableau}
-4\,
\begin{ytableau}
\hfil & \hfil & \hfil \\
\hfil & \hfil & \hfil
\end{ytableau}
\right) \cap\, c_{\mathrm{SM}}(X)\\
&=& 
\begin{ytableau}
\hfil & \hfil \\
\hfil 
\end{ytableau}
+3\,
\begin{ytableau}
\hfil & \hfil & \hfil\\
\hfil 
\end{ytableau}
+3 \,
\begin{ytableau}
\hfil & \hfil \\
\hfil & \hfil 
\end{ytableau}
+8 \,
\begin{ytableau}
\hfil & \hfil & \hfil \\
\hfil & \hfil 
\end{ytableau}
+5\,
\begin{ytableau}
\hfil & \hfil & \hfil \\
\hfil & \hfil & \hfil
\end{ytableau}.
\end{eqnarray*}
The last equality uses the formula for $c_{\mathrm{SM}}(X)$ from \cite{aluffi2009chern}, and the resulting formula checks with the CSM class computation in \cite{aluffi2009chern}.


\section{The locus \texorpdfstring{$\Omega_{\bm{\lambda}}$}{Omegapq} and its motivic Hirzebruch class}
\label{sec:Omega}
We discuss here a setting which is particularly relevant in the study of pointed Brill-Noether  varieties (\S\ref{sec:pBNvar}).
Consider the following maps of vector bundles over a variety~$X$:
\[
E_p \xrightarrow{\varphi} F_{q_1} \twoheadrightarrow F_{q_2} \dots \twoheadrightarrow F_{q_t}.
\]
The {Grassmann degeneracy locus}  is
\[
W_{\bm{\lambda}}:=\left\{x\in X \, : \, \dim \textrm{ker}(E_p \rightarrow F_{q_i})|_x \geq i \, \mbox{ for all $i$} \right\},
\]
with partition $\bm{\lambda}=(\lambda_1, \dots, \lambda_t)$ from \S\ref{sec:la} equal to
$\lambda_i := q_i-p+i$.
The locus $W_{\bm{\lambda}}$ is the  degeneracy locus $W_{\bt}$ corresponding to the triple $\bt=\left( \bm{k}, \bm{p}, \bm{q}\right)$ with 
$\bm{k}=(1,\dots,t)$, $\bm{p}=(p,\dots,p)$, and $\bm{q}=(q_1,\dots, q_t)$.
Its motivic Hirzebruch class is computed by Theorems \ref{thm:1Omegatildefrom1W} and \ref{thm:TyOmegatilde}.
Now consider the Grassmannian bundle $\bGr\left(t,E_p\right)$ on $X$ with tautological rank $t$ sub-bundle $\mathbb{S}$,
 and define its subvariety $\Omega_{\bm{\lambda}}$   by the conditions
\[
\dim \ker (\mathbb{S} \rightarrow F_{q_i})\geq i \quad \mbox{for  $1\leq i\leq t$}. 
\]
One has 
\[
\begin{tikzcd}
\Omega_{\bm{\lambda}} \arrow{d}[swap]{\phi} \arrow[hookrightarrow]{r}{\iota} \arrow{d} & \bGr(t,E_p) \arrow{d}{\pi}\\
W_{\bm{\lambda}} \arrow[hookrightarrow]{r} & X,
\end{tikzcd}
\]
and the fiber of $\phi$ over a point $x$ is 
\begin{equation}
\label{eq:fiberomega}
\left\{V\in \mathrm{Gr}\left(t,E_p|_x\right) \, : \, \dim \left( V\cap \mathrm{ker}\left( E_p\rightarrow F_{q_i}\right)|_x \right)\geq i \,\mbox{ for all $i$} \right\}.
\end{equation}

The aim of this section is to describe the (push-forward of the) motivic Hirzebruch class of $\Omega_{\bm{\lambda}}$ in terms of the motivic Hirzebruch class of loci of type $W_{\bm{\lambda}}$.
This is achieved in Proposition \ref{prop:1Omegafrom1W}. 
The (push-forward of the) motivic Hirzebruch class of $\Omega_{\bm{\lambda}}$ then follows after applying Theorems \ref{thm:1Omegatildefrom1W} and \ref{thm:TyOmegatilde} to compute the motivic Hirzebruch class of the loci~$W_{\bm{\lambda}}$.

\subsection{The stratification}
\label{ss:fiber}
We start by describing a stratification of $W_{\bm{\lambda}}$ induced by the projection $\phi\colon\Omega_{\bm{\lambda}} \rightarrow W_{\bm{\lambda}}$.
Consider the loci
\[
W_{\bm{\lambda}}^{\bm{\kappa}}:=\left\{x\in X \, : \, \dim \textrm{ker}\left(E_p \rightarrow F_{q_i}\right)|_x \geq i+\kappa_i  \, \mbox{ for all $i$}\right\} \subseteq W_{\bm{\lambda}},
\]
with $\bm{\kappa}=(0\leq \kappa_1\leq\dots \leq \kappa_t)$ such that $\bm{\lambda} + \bm{\kappa}$ is a partition.
The map $\phi$ is locally trivial precisely on the locally closed strata 
\[
\left(W_{\bm{\lambda}}^{\bm{\kappa}}\right)^\circ :=\left\{x\in X \, : \, \dim \textrm{ker}\left(E_p \rightarrow F_{q_i}\right)|_x = i+\kappa_i  \, \mbox{ for all $i$}\right\} \subseteq W_{\bm{\lambda}}^{\bm{\kappa}}.
\]
The fiber of $\phi$ over a general point $x$ in $W_{\bm{\lambda}}^{\bm{\kappa}}$ from \eqref{eq:fiberomega}
coincides with the Schubert variety $S_{\bm{\kappa}^c}$ in $\textrm{Gr}(t,E_p|_x)$ associated with the partition $\bm{\kappa}^c$ complementary to $\bm{\kappa}$ inside the $t\times(p-t)$ rectangle. For instance, when $t=3$, $p=7$, and $\bm{\kappa}=(1,1,2)$, one has $\bm{\kappa}^c = (3,3,2)$.

In the following, it is convenient to identify a weakly increasing sequence $\bm{\kappa}$ with the shape consisting of $\kappa_i$ boxes in the $i$th row, and note that the componentwise order is compatible with containment of shapes.

\begin{remark}
\label{rmk:lakavsla}
Fix $\bm{\kappa}=(0\leq \kappa_1\leq\dots \leq \kappa_t)$
such that $\bm{\lambda} + \bm{\kappa}$ is a partition.
We emphasize that the partition in \S\ref{sec:la} assigned to the degeneracy locus $W_{\bm{\lambda}}^{\bm{\kappa}}$ is not $\bm{\lambda} + \bm{\kappa}$,  but rather the partition with $1+\kappa_1$ parts equal to $\lambda_1+\kappa_1$, and $1+\kappa_2-\kappa_1$ parts equal to $\lambda_2+\kappa_2$,~etc. 
\end{remark}

\subsection{The sequence \texorpdfstring{$\bm{\kappa}^{\rm red}$}{kred}}
Fix a partition $\bm{\lambda}=(\lambda_1, \dots, \lambda_t)$ and $\bm{\kappa}=(0\leq \kappa_1\leq\dots\leq \kappa_t)$ such that $\bm{\lambda}+\bm{\kappa}$ is weakly decreasing.

\begin{definition}
\label{kred}
Given $\bm{\lambda}$ and $\bm{\kappa}$ as above, the sequence 
\[
\bm{\kappa}^{\rm red}=(0\leq \kappa^{\rm red}_1\leq \cdots\leq \kappa^{\rm red}_t)
\]
is defined as the minimal  sequence in componentwise order with 
$\kappa^{\rm red}_t=\kappa_t$ and 
\[
\kappa^{\rm red}_i=\kappa_i \qquad \mbox{when $\kappa_{i+1}+\lambda_{i+1}< \kappa_i + \lambda_i$, for $i<t$.}
\]
\end{definition}

For instance, for $\bm{\lambda}=(4,4,1,1)$ and $\bm{\kappa}=(1,1,3,3)$ one has $\bm{\kappa}^{\rm red}=(0,1,1,3)$. 
See Example \ref{ex:shaded} for a graphical representation of some sequences $\bm{\kappa}$ and corresponding~$\bm{\kappa}^{\rm red}$.

\begin{lemma}
\label{lemma:kred}
Given $\bm{\lambda}$ and $\bm{\kappa}$ as above,
$\bm{\kappa}^{\rm red}$ is the smallest shape inside  $\bm{\kappa}$ which is not contained  in any of the shapes $\bm{\epsilon} <\bm{\kappa}$ such that $\bm{\lambda}+\bm{\epsilon}$ is weakly decreasing.
\end{lemma}

\begin{proof}
We first prove that 
if $\bm{\epsilon}$ contains $\bm{\kappa}^{\rm red}$ and $\bm{\lambda}+\bm{\epsilon}$ is weakly decreasing, then $\bm{\epsilon}=\bm{\kappa}$. One has necessarily $\epsilon_t=\kappa_t$. Next, fix $i<t$ such that $\kappa^{\rm red}_i<\kappa_i$, and assume that we have already settled that $\epsilon_j=\kappa_j$ for $j>i$. By definition, since $\kappa^{\rm red}_i<\kappa_i$, one has $\kappa_{i+1}+ \lambda_{i+1} = \kappa_i + \lambda_i$. Since we want that $\lambda_i + \epsilon_i \geq \lambda_{i+1} + \epsilon_{i+1}=\lambda_{i+1} +\kappa_{i+1}$, then necessarily $\epsilon_i=\kappa_i$.

Finally, we show that 
any smaller shape $\bm{\kappa}'<\bm{\kappa}^{\rm red}$ is contained inside some $\bm{\epsilon}<\bm{\kappa}$  such that $\bm{\lambda}+\bm{\epsilon}$ is weakly decreasing. If $\kappa'_t<\kappa_t=\kappa^{\rm red}_t$, then $\bm{\kappa}'$ is contained inside 
\[
\bm{\epsilon}=(\min \{\kappa_1, \kappa_t-1 \},\dots, \min \{\kappa_{t-1}, \kappa_t-1 \},\kappa_t-1).
\]
Next, fix $i<t$ such that $\kappa^{\rm red}_i> \kappa^{\rm red}_{i-1}$, or $i=1$ and $\kappa_1>0$. By the definition of $\bm{\kappa}^{\rm red}$, this implies that $\kappa^{\rm red}_i=\kappa_i$ and $\kappa_{i+1}+ \lambda_{i+1} < \kappa_i + \lambda_i$. Consider $\bm{\kappa}'$ defined as $\kappa'_i=\kappa^{\rm red}_i-1$, and  $\kappa'_j=\kappa^{\rm red}_j$ for $j\not=i$.
Then $\bm{\kappa}'$ is contained inside
\[
\bm{\epsilon} = (\min \{\kappa_1, \kappa_i-1 \},\dots, \min \{\kappa_{i-1}, \kappa_i-1 \},\kappa_i-1,\kappa_{i+1},\dots,\kappa_t).
\]
The condition $\kappa_{i+1}+ \lambda_{i+1} < \kappa_i + \lambda_i$ guarantees that $\bm{\lambda}+\bm{\epsilon}$ is weakly decreasing.
\end{proof}

\subsection{The motivic Hirzebruch class of \texorpdfstring{$\Omega_{\bm{\lambda}}$}{Omegapq}} 

\begin{proposition}
\label{prop:1Omegafrom1W}
With assumptions as in \S\ref{sec:assumptions}, one has 
\[
(\pi \iota)_* \, T_y\left( \Omega_{\bm{\lambda}}\right) = \sum_{\bm{\kappa}\geq 0} d_{\bm{\kappa}} \,\iota_*\,T_y\left(W_{\bm{\lambda}}^{\bm{\kappa}}\right)
\]
where the sum is over $\bm{\kappa}=(0\leq \kappa_1\leq\dots\leq \kappa_t)$ such that $\bm{\lambda}+\bm{\kappa}$ is weakly decreasing~and
\begin{equation}
\label{dk}
d_{\bm{\kappa}}:= \mathop{\sum_{\bm{\kappa}'=(\kappa'_1\leq \cdots \leq \kappa'_t)}}_{\bm{\kappa}^{\rm red} \leq \bm{\kappa}' \leq \bm{\kappa}} (-y)^{|\bm{\kappa}'|}.
\end{equation}
\end{proposition}

\begin{example}
When $y=-1$, the topological Euler  characteristic of the fibers \eqref{eq:fiberomega} of $\phi$ and consequently the coefficients $d_{\bm{\kappa}}$ are computed as follows.

For $\bm{\kappa}=(0\leq \kappa_1\leq \cdots \leq \kappa_l)$, the topological Euler  characteristic of the Schubert variety $S_{\bm{\kappa}^c}$  in $\textrm{Gr}(l,\mathbb{C}^p)$ associated with the partition complementary to $\bm{\kappa}$ inside the $l\times(p-l)$ rectangle can be computed as
\begin{equation}
\label{chiSkappac}
\chi \left(S_{\bm{\kappa}^c} \right) = p \left(\bm{\kappa} \right) := \left| {\kappa_{l+1-j} + l-i+1 \choose 1+j-i } \right|_{1\leq, i,j \leq l}.
\end{equation}
Indeed, $\chi \left(S_{\bm{\kappa}^c} \right)$ coincides with the number  of partitions  inside $(\kappa_l,\dots,\kappa_1)$, or equivalently, the number of shapes inside the shape $\bm{\kappa}$. This equals the number of non-intersecting lattice paths 
\begin{align*}
\mbox{from the points} & \qquad (0,0), (1,0), \dots, (l-1,0) \\
\mbox{to the points} & \qquad (1,\kappa_l+l-1), (2,\kappa_{l-1}+l-2),\dots, (l,\kappa_1)
\end{align*}
(see e.g., \cite{bhstackex}). From the Gessel-Viennot formula \cite{gv}, this number equals the above $p(\bm{\kappa})$, hence the first equality in \eqref{chiSkappac}. 

By definition \eqref{dk}, when $y=-1$ the coefficient $d_{\bm{\kappa}}$ counts the number of shapes inside the shape $\bm{\kappa}$ and  containing the shape $\bm{\kappa}^{\rm red}$, that is, the number of shapes in $\bm{\kappa}\setminus \bm{\kappa}^{\rm red}$.
Now Lemma \ref{lemma:kred} implies that $\bm{\kappa}\setminus \bm{\kappa}^{\rm red}$ is a disjoint union of shorter weakly increasing sequences, say $\bm{\kappa}^{(1)}, \bm{\kappa}^{(2)},\dots$. The number of shapes in $\bm{\kappa}\setminus \bm{\kappa}^{\rm red}$ is equal to the product of
the number of shapes inside each $\bm{\kappa}^{(i)}$. It follows that $d_{\bm{\kappa}}$ can be computed as the product of the quantities $p\left( \bm{\kappa}^{(i)}\right)$.

For instance, when $\bm{\lambda}=(4,4,1,1)$ and $\bm{\kappa}=(1,1,3,3)$, one has $\bm{\kappa}^{\rm red}=(0,1,1,3)$, hence $\bm{\kappa}\setminus \bm{\kappa}^{\rm red}$ is the disjoint union of the length one sequences $(1)$ and $(2)$ (see also how the white tiles representing $\bm{\kappa}\setminus \bm{\kappa}^{\rm red}$ are indeed union of disjoint shapes in Example \ref{ex:shaded}) and indeed $d_{\bm{\kappa}}=p((1))\cdot p((2))=2\cdot 3=6$.
\end{example}

\begin{example}
\label{ex:dkappaalllaeq}
When $\lambda_1=\cdots=\lambda_t$, one has necessarily $\kappa_1=\cdots=\kappa_t$, hence we have $\bm{\kappa}=(\kappa,\dots,\kappa)$. 
For $\bm{\kappa}=(\kappa,\dots,\kappa)=:\kappa^t$, one has $\bm{\kappa}^{\rm red}=(0,\dots,0,\kappa)$ and thus $\bm{\kappa}\setminus \bm{\kappa}^{\rm red}\cong \kappa^{t-1}$. 
For $y=-t^2$, the right-hand side of \eqref{dk} is $t^{2|\bm{\kappa}^{\rm red}|} P_{(\kappa, t-1+\kappa)}(t)$, hence we have
\[
d_{\bm{\kappa}} = t^{2\kappa} P_{(\kappa, t-1+\kappa)}(t), 
\]
where
$P_{(\kappa, t-1+\kappa)}(t)$ is the Poincar\'e polynomial  of the Grassmannian $\textrm{Gr}(\kappa,\mathbb{C}^{t-1+\kappa})$.
Since $P_{(\kappa, t-1+\kappa)}(t)$ is given by the $t^2$-binomial coefficient 
\[
P_{(\kappa, t-1+\kappa)}(t) = \left[\begin{array}{c} t-1+\kappa \\ \kappa\end{array}\right]_{t^2} := 
\frac{\prod_{i=1}^{t-1+\kappa}\, \left(1-t^{2i}\right)}{\prod_{i=1}^{\kappa} \left(1-t^{2i}\right)\prod_{i=1}^{t-1} \left(1-t^{2i}\right)},
\]
we have
\begin{equation*}
(\pi \iota)_* \,T_y\left( \Omega_{\bm{\lambda}}\right) = \sum_{\kappa\geq 0} 
(-y)^{\kappa} \left[\begin{array}{c}t-1+\kappa \\ \kappa\end{array}\right]_{-y}
\, \iota_*\,T_y\left(W_{\bm{\lambda}}^{\bm{\kappa}}\right)
\end{equation*}
where $\bm{\kappa}=\kappa^t$.
For $y=-1$, the coefficients in the sum specialize to the binomial coefficients ${t-1+\kappa \choose \kappa}$.
\end{example}

\begin{example}
\label{ex:shaded}
When $\bm{\lambda} =(4,4,1,1)$, one has that $(\pi \iota)_* \,T_y\left(\Omega_{\bm{\lambda}}\right)$ equals
\begin{align*}
 &1 
- y(1-y) \,
\begin{ytableau}
\hfil\\
*(gray)
\end{ytableau}
+ y^2\left(1-y+y^2\right) \,
\begin{ytableau}
\hfil & \hfil\\
*(gray) & *(gray)
\end{ytableau}
- y^3\left(1-y+y^2-y^3\right) \,
\begin{ytableau}
\hfil & \hfil & \hfil\\
*(gray) & *(gray) & *(gray)
\end{ytableau}
\\
&{} 
- y^3(1-y) \,
\begin{ytableau}
\hfil\\
*(gray)\\
*(gray)\\
*(gray)
\end{ytableau}
+ y^4(1-y)^2 \,
\begin{ytableau}
\none & \hfil\\
\none & *(gray)\\
\hfil & *(gray)\\
*(gray) & *(gray)
\end{ytableau}
- y^5 \left(1-y+y^2\right)(1-y) \,
\begin{ytableau}
\none & \none & \hfil\\
\none & \none & *(gray)\\
\hfil & \hfil & *(gray)\\
*(gray) & *(gray) & *(gray)
\end{ytableau}
\\
&{} + y^4\left(1-y+2y^2-3y^3+3y^4-2y^5+y^6\right) \,
\begin{ytableau}
\none & \none & \none & \hfil\\
\none & \none & \none & \hfil\\
\hfil & \hfil & \hfil  & \hfil\\
*(gray) & *(gray) & *(gray) & *(gray)
\end{ytableau}
\\
&{}+y^6\left(1-y+y^2\right) \,
\begin{ytableau}
\hfil & \hfil\\
*(gray) & *(gray)\\
*(gray) & *(gray)\\
*(gray) & *(gray)
\end{ytableau}
-y^7(1-y)\left(1-y+y^2\right) \,
\begin{ytableau}
\none & \hfil & \hfil\\
\none & *(gray) & *(gray)\\
\hfil & *(gray) & *(gray)\\
*(gray) & *(gray) & *(gray)
\end{ytableau}\\
&{}+ y^8\left(1-y+y^2\right)^2 \,
\begin{ytableau}
\none & \none & \hfil & \hfil\\
\none & \none & *(gray) & *(gray)\\
\hfil & \hfil & *(gray) & *(gray)\\
*(gray) & *(gray) & *(gray) & *(gray)
\end{ytableau}\\
&{}-y^5\left(1-y+2y^2-3y^3+4y^4-5y^5+5y^6-4y^7+2y^8-y^9\right) \,
\begin{ytableau}
\none & \none & \none & \hfil & \hfil\\
\none & \none & \none & \hfil & \hfil\\
\hfil & \hfil & \hfil  & \hfil & \hfil\\
*(gray) & *(gray) & *(gray) & *(gray) & *(gray)
\end{ytableau}
\cdots.
\end{align*}
Here for simplicity, the shape $\bm{\kappa}$ with $\kappa_i$ boxes in the $i$th row stands for $\iota_*\,T_y\left(W_{\bm{\lambda}}^{\bm{\kappa}}\right)$. 
Inside each $\bm{\kappa}$, the shape $\bm{\kappa}^{\rm red}$ is shaded. Note how the complement of $\bm{\kappa}^{\rm red}$ in each $\bm{\kappa}$ is a disjoint union of shapes.
\end{example}

\begin{proof}[Proof of Proposition \ref{prop:1Omegafrom1W}]
Write
\[
(\pi \iota)_* T_y\left( \Omega_{\bm{\lambda}}\right) = \sum_{\bm{\kappa}\geq 0} d_{\bm{\kappa}}\,\iota_*\, T_y\left(W_{\bm{\lambda}}^{\bm{\kappa}}\right)
\]
for some coefficients $d_{\bm{\kappa}}$.
One has $W_{\bm{\lambda}}^{\bm{\epsilon}} \supset W_{\bm{\lambda}}^{\bm{\kappa}}$ exactly when $\bm{\epsilon}$ is contained inside $\bm{\kappa}$, that is, $\epsilon_i \leq \kappa_i$ for all $i$.
From \S \ref{ss:fiber}, the fiber of $\phi$ over a general point in $W_{\bm{\lambda}}^{\bm{\kappa}}$ is the Schubert variety $S_{\bm{\kappa}^c}$  in $\textrm{Gr}(t,\mathbb{C}^p)$ associated with the partition $\bm{\kappa}^c$ complementary to $\bm{\kappa}$ inside the $t\times(p-t)$ rectangle. Hence from Lemma \ref{lemma:Fibra},  we have
\[
d_{\bm{\kappa}}= \chi_y (S_{\bm{\kappa}^c})-\sum_{\bm{\epsilon} < \bm{\kappa}} d_{\bm{\epsilon}}
\] 
where the sum is over $\bm{\epsilon}$ such that $\bm{\lambda}+\bm{\epsilon}$ is weakly decreasing, and $\epsilon_i \leq \kappa_i$ for all $i$, with at least one strict inequality. 

The Schubert variety $S_{\bm{\kappa}^c}$ is the union of the Schubert cells corresponding to
 (weakly decreasing) partitions contained inside $(\kappa_t,\dots,\kappa_1)$, or equivalently, shapes inside the shape $\bm{\kappa}$.
The sum
\[
\sum_{\bm{\epsilon}< \bm{\kappa}} d_{\bm{\epsilon}}
\]
is the sum of the Hirzebruch $\chi_y$-genera of Schubert cells corresponding to shapes contained inside some $\bm{\epsilon}<\bm{\kappa}$ such that $\bm{\lambda}+\bm{\epsilon}$ is weakly decreasing.
From Lemma \ref{lemma:kred},
it follows that $d_{\bm{\kappa}}$ is the sum of the Hirzebruch $\chi_y$-genera of Schubert cells corresponding to shapes contained inside $\bm{\kappa}$ and containing $\bm{\kappa}^{\rm red}$. Since a Schubert cell corresponding to a shape $\bm{\kappa}'$ is isomorphic to the affine space $\mathbb{A}^{|\bm{\kappa}'|}_\mathbb{C}$, its Hirzebruch $\chi_y$-genus is $(-y)^{|\bm{\kappa}'|}$ (Example \ref{ex:TyPnAn}). The statement follows.
\end{proof}


\section{Pointed Brill-Noether varieties}
\label{sec:pBNvar}
For a smooth pointed curve $(C,P)$ and a  sequence  \mbox{$\bm{a}: 0\leq a_0 < \cdots <a_r \leq d$}, the \textit{pointed Brill-Noether variety of line bundles} $W^{\bm{a}}_d(C, P)$ is defined as
\[
W^{\bm{a}}_d(C, P) := \left\{ L\in \mathrm{Pic}^d(C) \, | \, h^0(C, L\otimes \mathscr{O}_C(-a_iP))\geq r+1-i \,\mbox{ for all $i$}\right\}.
\]
The \textit{pointed Brill-Noether variety of linear series} $G^{\bm{a}}_d(C, P)$ is defined as
\[
G^{\bm{a}}_d(C, P) := \left\{ (L,V) \, \Big| \, 
\begin{array}{l}
L\in \mathrm{Pic}^d(C), \, V\subseteq H^0(C, L), \, \dim (V)=r+1, \mbox{ and}\\ 
\dim \left( V\cap H^0(C, L\otimes \mathscr{O}_C(-a_iP))\right)\geq r+1-i \,\mbox{ for all $i$}
\end{array}
\right\}.
\]

The variety $W^{\bm{a}}_d(C, P)$ has  the structure of a Grassmannian degeneracy locus in $\mathrm{Pic}^d(C)$, and the variety $G^{\bm{a}}_d(C, P)$ is of type $\Omega_{\bm{\lambda}}$, as in \S\ref{sec:Omega}. We briefly review this in \S\ref{sec:BNconstruction},   we verify  the assumptions from \S\ref{sec:assumptions} in \S\ref{sec:assumptionsBN}, and  apply Theorems \ref{thm:1Omegatildefrom1W} and \ref{thm:TyOmegatilde} to compute the motivic Hirzebruch class of pointed Brill-Noether varieties in \S\ref{sec:motrefBN}. Finally, we conclude with some examples.

\subsection{The construction} 
\label{sec:BNconstruction}
Choose a positive integer $n$ large enough so that line bundles of degree $d+n$ are non-special, that is, $n\geq 2g-1-d$. Fix a Poincar\'e line bundle $\mathscr{L}$ on $C\times \mathrm{Pic}^d(C)$, normalized so that $\mathscr{L}|_{\{P\}\times \mathrm{Pic}^d(C)}$ is trivial. Consider the following vector bundles on $\mathrm{Pic}^d(C)$:
\begin{align*}
\mathscr{E} &:= \left(\pi_2\right)_* \left( \mathscr{L}\otimes\pi_1^*\mathscr{O}_C(nP) \right),\\
\mathscr{F}_i &:= \left(\pi_2\right)_* \left( \mathscr{L}\otimes\pi_1^*\mathscr{O}_{(n+a_{r+1-i})P} \right) \qquad \mbox{for $1\leq i \leq r+1$}.
\end{align*}
Here $\pi_1$ and $\pi_2$ are the projections from $C\times \mathrm{Pic}^d(C)$ to $C$ and $\mathrm{Pic}^d(C)$, respectively. One computes
\begin{align}
\begin{split}
\label{eq:pqi}
p &:= \mathrm{rank}\left( \mathscr{E}\right) = d+n-g+1, \\
q_i & := \mathrm{rank}\left( \mathscr{F}_i\right) = n+a_{r+1-i} \qquad \mbox{for $1\leq i \leq r+1$}.
\end{split}
\end{align}
There are natural maps
\[
\mathscr{E} \rightarrow \mathscr{F}_1 \twoheadrightarrow  \mathscr{F}_2 \cdots \twoheadrightarrow \mathscr{F}_{r+1}
\]
and  $W^{\bm{a}}_d(C, P)$ is the Grassmannian degeneracy locus  with   partition $\bm{\lambda}=(\lambda_1,\dots, \lambda_{r+1})$ from \S\ref{sec:la} equal to
\begin{equation}
\label{eq:laBN}
\lambda_i := g-d+r + a_{r+1-i} - (r+1-i) \qquad \mbox{for $1\leq i \leq r+1$.}
\end{equation}
One has $c(\mathscr{F}_i)=0$ for all $i$, and all  classes $c(i)$ from \S\ref{sec:cidef} are equal to $c=c(-\mathscr{E})=\mathrm{e}^\theta$ in $H^*(\mathrm{Pic}^d(C))$, where $\theta$ is the cohomology class of the theta divisor \cite[\S VIII]{MR770932}. Finally, note that $T_y\left(\mathrm{Pic}^d(C)\right)=1$, as Abelian varieties have trivial tangent bundles.

\subsection{Dimension and singular locus} 
\label{sec:assumptionsBN}
The \textit{one-pointed Brill-Noether Theorem} \cite[\S 1]{MR910206} says that for a \textit{general} smooth pointed curve $(C,P)$: (i) the varieties $W^{\bm{a}}_d(C, P)$ and $G^{\bm{a}}_d(C, P)$ are non-empty if and only if $g\geq\sum_{i=0}^r \mathrm{max}\{0, g-d+r+a_i-i\}$, 
and (ii) when non-empty, $G^{\bm{a}}_d(C, P)$ has dimension equal to the \textit{one-pointed Brill-Noether number} $\rho(g,r,d,\bm{a}):=g-\sum_{i=0}^r(g-d+r+a_i-i)=g-|\bm{\lambda}|$; the same holds for $W^{\bm{a}}_d(C, P)$ when  $\rho(g,r,d,\bm{a})\leq g$.

The proof in \cite{MR910206} uses degenerations to singular curves. However, explicit examples of \textit{smooth} pointed curves of any genus verifying the one-pointed Brill-Noether Theorem and defined over $\mathbb{Q}$ were provided in \cite{ft2}.

Furthermore, the \textit{one-pointed Gieseker-Petri Theorem} \cite{chan2019gieseker} characterizes the smooth locus of $G^{\bm{a}}_d(C, P)$ for a \textit{general} smooth pointed curve $(C,P)$, and implies that  the singular locus of $G^{\bm{a}}_d(C, P)$ is contained in the locus of linear series with excess vanishing  at $P$. Since for $\rho(g,r,d,\bm{a})\leq g$ the forgetful map $\pi\colon G^{\bm{a}}_d(C, P) \rightarrow W^{\bm{a}}_d(C, P)$ is an isomorphism when restricted over the  locus 
\[
W^{\bm{a}}_d(C, P)^\circ := \left\{ L\in \mathrm{Pic}^d(C) \, | \, h^0(C, L\otimes \mathscr{O}_C(-a_iP))= r+1-i \,\mbox{ for all $i$}\right\} \subseteq  W^{\bm{a}}_d(C, P)
\]
and the smooth locus of $G^{\bm{a}}_d(C, P)$ as described in \cite{chan2019gieseker} includes $\pi^{-1}\left( W^{\bm{a}}_d(C, P)^\circ \right)$, it follows that $W^{\bm{a}}_d(C, P)^\circ$ is smooth for a {general}  $(C,P)$ and $\rho(g,r,d,\bm{a})\leq g$. 

The stratification of $W^{\bm{a}}_d(C, P)$ from \S\ref{sec:filtr} consists of strata whose closures are themselves pointed Brill-Noether varieties of line bundles inside $\mathrm{Pic}^d(C)$. Namely, the strata are
\begin{equation}
\label{eq:Wa+}
W^{\bm{a^+}}_d(C, P) :=\left\{ L\in \mathrm{Pic}^d(C) \, | \, h^0(C, L\otimes \mathscr{O}_C(-a_iP))\geq k^+_{r+1-i} \,\mbox{ for all $i$}\right\}
\end{equation}
for $\bm{k^+}\geq\bm{k}=(1,\dots, r+1)$. (Note that some of the conditions defining $W^{\bm{a^+}}_d(C, P)$ may be redundant.)

In particular, $W^{\bm{a}}_d(C, P)$ satisfies the transversality assumption from \S\ref{sec:assumptions} for a general $(C,P)$ and $\rho(g,r,d,\bm{a})\leq g$. 

While in this paper we only treat the one-pointed case, these results are also known more generally for two-pointed Brill-Noether varieties with  appropriate changes, see \cite{MR910206, o, chan2019gieseker}.

\subsection{Motivic Hirzebruch class of pointed Brill-Noether varieties}
\label{sec:motrefBN}
For the pointed Brill-Noether variety $W_d^{\bm{a}}(C,P)$ with $\rho(g,r,d,\bm{a})\leq g$, the (push-forward of the) motivic Hirzebruch class of its resolution  from Theorem \ref{thm:TyOmegatilde} equals 
\begin{equation}
\label{eq:TyOmegatilde}
A_{\bm{\lambda}}:=\prod_{i=1}^{r+1} \dfrac{\prod_{k=1}^p Q_y(\alpha_k+R_i)}{\prod_{j=1}^{i-1}Q_y(R_i-R_j)} \left(Q_y(R_i)\right)^{-q_i} \left|c_{\lambda_i+j-i}\right|_{1\leq i,j\leq r+1}
\end{equation}
where $\alpha_1,\dots,\alpha_p$ are the Chern roots of the vector bundle $\mathscr{E}$. These satisfy
\begin{equation}
\label{eq:alpharel}
\sum_{k=1}^p \alpha_k = -\theta \qquad \mbox{and} \qquad \sum_{k=1}^p \alpha_k^i =0 \quad \mbox{for $i>1$}
\end{equation}
\cite[pg.~336]{MR770932}. The expression \eqref{eq:TyOmegatilde} is symmetric in the $\alpha_k$, hence after expanding and using \eqref{eq:alpharel}, can be rewritten as a polynomial in $\theta$ with coefficients in $H^*\left(\mathrm{Pic}^d(C) \right)[y]$.
For instance, the first terms of $\prod_{k=1}^p Q_y(\alpha_k+R_i)$ are
\[
Q_y(R_i)^p - \frac{1-y}{2} \,\theta \, Q_y(R_i)^{p-1}
 - \frac{(1+y)^2}{6} \theta \, R_i\, Q_y(R_i)^{p-1} + \frac{(1-y)^2}{8}\, \theta^2 \, Q_y(R_i)^{p-2} + \dots.
\]
When $y=-1$, the expression $A_{\bm{\lambda}}$ reduces as follows:
\[
\prod_{i> j}  \frac{1}{1+R_i-R_j} 
\left| \frac{(1+R_i)^{-\lambda_i+i}}{\mathrm{e}^{\frac{\theta}{1+R_i}}} c_{\lambda_i+j-i} \right|_{1\leq i,j\leq k_{r+1}}.
\]

Theorems \ref{thm:1Omegatildefrom1W} and \ref{thm:TyOmegatilde} imply:

\begin{corollary}
\label{cor:BNW}
For a general smooth pointed curve $(C,P)$ of genus $g$ and for a sequence $\bm{a}: 0\leq a_0 < \cdots <a_r \leq d$ such that $\rho(g,r,d,\bm{a})\leq g$, one has
\[
 A_{\bm{\lambda}} = \sum_{\bm{k^+}} (-y)^{|\bm{k^+}|-|\bm{k}|}  \, \iota_* \,T_y\left( W^{\bm{a^+}}_d(C, P) \right) \qquad \mbox{in $H^*\left( \mathrm{Pic}^d(C)\right)[y]$,}
\]
where  
the sum is over the set of weakly increasing sequences $\bm{k^+}\geq\bm{k}=(1,\dots, r+1)$, and for a given  $\bm{k^+}$, the locus $W^{\bm{a^+}}_d(C, P)\subseteq  W^{\bm{a}}_d(C, P)$ is as in \eqref{eq:Wa+}. 

By the inclusion-exclusion principle, this determines the class $\iota_* \,T_y\left( W^{\bm{a}}_d(C, P) \right)$ in $H^*\left( \mathrm{Pic}^d(C)\right)[y]$ in terms of classes of type 
$A_{\bm{\lambda}}$.
\end{corollary}

Similarly, Proposition \ref{prop:1Omegafrom1W} implies:

\begin{corollary}
\label{cor:BNG}
For a general smooth pointed curve $(C,P)$ of genus $g$ and for a sequence $\bm{a}: 0\leq a_0 < \cdots <a_r \leq d$ such that $\rho(g,r,d,\bm{a})\leq g$, one has
\[
(\pi \iota)_* \, T_y \left(G^{\bm{a}}_d(C, P)\right) = \sum_{\bm{\kappa}\geq 0} d_{\bm{\kappa}}\, \iota_*\,T_y\left(W^{\bm{a}(\bm{\kappa)}}_d(C, P)\right) \qquad \mbox{in $H^*\left( \mathrm{Pic}^d(C)\right)[y]$,}
\]
where the sum is over $\bm{\kappa}=(0\leq \kappa_1\leq\dots\leq \kappa_{r+1})$ such that $\bm{\lambda}+\bm{\kappa}$ is weakly decreasing, the coefficients $d_{\bm{\kappa}}$ are as in \eqref{dk}, and $\bm{a}(\bm{\kappa})$ is the strictly increasing sequence
\[
\bm{a}(\bm{\kappa})=\left(a_0, a_0+1, \dots, a_0+\kappa_{r+1}-\kappa_r,  \dots,  a_r, a_r +1,\dots, a_r+\kappa_1\right).
\]
\end{corollary}

The  expression  for $\bm{a}(\bm{\kappa)}$ is obtained from \eqref{eq:laBN} using the partition $\lambda$  assigned to $W^{\bm{a}(\bm{\kappa)}}_d(C, P)$ as in Remark \ref{rmk:lakavsla}.
Examples are studied below.

\subsection{The curve case}
When $(C,P)$ is general of genus $g\geq 1$ and $W^{\bm{a}}_d(C, P)$ is one-dimensional, the stratification from \S \ref{sec:filtr} consists of only the stratum with $\bm{k^+}=(1,\dots,r+1)$. 
Hence the curve $W^{\bm{a}}_d(C, P)$ is necessarily smooth. 
Expanding formula \eqref{eq:TyOmegatilde}, one has that $T_y\left(W^{\bm{a}}_d(C, P) \right)$ is 
\begin{equation}
\label{eq:csmBNcurve}
 \left(1 +\frac{1}{2}\left((r+1)\theta+ \sum_{k=1}^{r+1} (\lambda_k-r-2+k)T_k \right)(y-1) \right)\left| c_{\lambda_i+j-i} \right|_{1\leq i,j\leq r+1}.
\end{equation}
One computes
\begin{align*}
\theta \left| c_{\lambda_i+j-i} \right|_{1\leq i,j\leq r+1} &= g! \left| \frac{1}{(\lambda_i+j-i)!} \right|_{1\leq i,j\leq r+1} , \\
T_k \left| c_{\lambda_i+j-i} \right|_{1\leq i,j\leq r+1} &= g!\left| \frac{1}{(\lambda_i+\delta_{i,k}+j-i)!} \right|_{1\leq i,j\leq r+1}
\end{align*}
where $\delta_{i,k}$ is the Kronecker delta: $\delta_{i,k}=1$ for $i=k$, and $\delta_{i,k}=0$ otherwise.
The determinants can be computed via the following application of the Vandermonde identity:
\[
 \left|  \frac{1}{(l_i+j-i)!} \right|_{1\leq i,j\leq r+1} 
= \frac{\prod_{1\leq i<j \leq r+1} (l_i-l_j+j-i)}{\prod_{i=1}^{r+1} (l_i+r+1-i)!}.
\]
Using the combinatorial identity
\[
(r+1) \left| \frac{1}{(\lambda_i+j-i)!} \right|_{1\leq i,j\leq r+1}
 = \sum_{k=1}^{r+1}(\lambda_k+r+2-k)\left| \frac{1}{(\lambda_i+\delta_{i,k}+j-i)!} \right|_{1\leq i,j\leq r+1},
\]
the top degree part in \eqref{eq:csmBNcurve} gives
\[
\chi_{y} \left( W^{\bm{a}}_d(C, P) \right)  =(y-1)\,g! \sum_{k=1}^{r+1} \lambda_k \left| \frac{1}{(\lambda_i+\delta_{i,k}+j-i)!} \right|_{1\leq i,j \leq r+1}.
\]
For the curve $G^{\bm{a}}_d(C, P)$, the stratification in \S\ref{ss:fiber} consists of a single non-empty stratum, thus $\chi_{y} \left( G^{\bm{a}}_d(C, P) \right)=\chi_{y} \left( W^{\bm{a}}_d(C, P) \right)$.
As a check, for $y=0$ one recovers the \textit{holomorphic} Euler characteristic computed in \cite[\S 4.2]{act}.

\subsection{The surface case when $\lambda_1=\cdots=\lambda_{r+1}$} 
\label{sec:BNsurfalleqlas}
Let us consider  the surface case in the \textit{classical} Brill-Noether setting, that is, with no special ramification required at the marked point. Thus one has $a_i=i$ for $0\leq i\leq r$, and $\lambda_1=\cdots=\lambda_{r+1}=g-d+r=:\lambda$. In this case, one simplifies the notation as $W^r_d(C):=W^{\bm{a}}_d(C, P)$, and similarly, $G^r_d(C):=G^{\bm{a}}_d(C, P)$. The degree-two operators with non-zero action are:
\begin{align}
\label{TT}
\begin{split}
R_1^2 \left| c_{\lambda+j-i} \right|_{1\leq i,j \leq r+1} &= \frac{(r+1)(r+2)}{2(\lambda+r+1)(\lambda+r+2)}\,g! \prod_{i=0}^r \frac{i!}{(\lambda+i)!},\\
R_1 R_2 \left| c_{\lambda+j-i} \right|_{1\leq i,j \leq r+1}  &= \frac{r(r+1)}{2(\lambda+r)(\lambda+r+1)}\,g! \prod_{i=0}^r \frac{i!}{(\lambda+i)!}, \\
R_2^2  \left| c_{\lambda+j-i} \right|_{1\leq i,j \leq r+1} &= - R_1 R_2 \left| c_{\lambda+j-i} \right|_{1\leq i,j \leq t},\\
\theta \,R_1 \left| c_{\lambda+j-i} \right|_{1\leq i,j \leq r+1} &= \frac{r+1}{\lambda+r+1} \, g! \prod_{i=0}^r \frac{i!}{(\lambda+i)!},\\
\theta^2 \left| c_{\lambda+j-i} \right|_{1\leq i,j \leq r+1}  &= g! \prod_{i=0}^r \frac{i!}{(\lambda+i)!}.
\end{split}
\end{align}
We are now ready to  prove  Corollary \ref{cor:TyWrd}:

\begin{proof}[Proof of Corollary \ref{cor:TyWrd}]
The stratification from \S \ref{sec:filtr} consists of two strata: the full-dimen\-sional stratum with $\bm{k^+}=\bm{k}=(1,2,\dots,r+1)$, and the codimension two stratum with $\bm{k^+}=(2,2,\dots,r+1)$  whose class is equal to $R_1R_2 \left|c_{\lambda+j-i}\right|_{1\leq i,j\leq r+1}$.
From Corollary \ref{cor:BNW}, the motivic Hirzebruch class is given by
\begin{equation}
\label{eq:Tyinter}
T_y\left( W^r_d(C) \right) = A_{\bm{\lambda}} -\chi_y\left(\mathbb{A}^1_{\mathbb{C}}\right) \, R_1R_2 \left|c_{\lambda+j-i}\right|_{1\leq i,j\leq r+1}.
\end{equation}
As in the proof of Theorem \ref{thm:1Omegatildefrom1W},  the affine space $\mathbb{A}^1_{\mathbb{C}}$ here coincides with the maximal Schubert cell inside the Schubert variety given by the generic fiber of $\widetilde{\Omega}_{\bt}$ over the codimension-two stratum in $W^r_d(C)$. 

One has  $\chi_{y}\left(\mathbb{A}^n_{\mathbb{C}}\right)=(-y)^n$ (Example \ref{ex:TyPnAn}). 
On a surface, the power series $Q_y(\alpha)$ restricts as:
\[
Q_y(\alpha) = 1 + \frac{1}{2}\,\alpha(1-y) + \frac{1}{12}\,\alpha^2(1+y)^2.
\]
The resulting expansion of \eqref{eq:Tyinter} is
\begin{eqnarray*}
T_y\left( W^r_d(C) \right) &=&\bigg( 1+ \frac{1}{2} \left((\lambda-r-1)R_1+(\lambda-r)    R_2+ (r+1)\theta\right)(y-1)\\
&&{}+\frac{1}{24} \bigg( 
\big( \left( 3( \lambda - r)^2 +7r  -5\lambda +2 \right) (y-1)^2 \\
&&{}   + 8(1-r -\lambda )y \big)R_1^2\\
   && {}+2  \left( \left(3( \lambda-r )( \lambda-r -1) -1\right)(y-1)^2 +20y \right)R_1 R_2\\
   && {}+\left( 3 (\lambda-r)^2 (y-1)^2+(\lambda+r-2)\left((y-1)^2 -8y \right)\right) R_2^2\\
   &&{}  +2 \left( (3\lambda(r+1)-3r(r+2)-2)(y-1)^2 -8y \right)\theta R_1\\ 
   &&{} +2 \left( (3\lambda(r+1)-3r(r+1)+1 )(y-1)^2-8y \right)\theta R_2\\ 
   &&{} +3(r+1)^2 (y-1)^2\theta ^2 \bigg)\bigg) \left| c_{\lambda+j-i}\right|_{1\leq i,j\leq r+1}.
\end{eqnarray*}
Using \eqref{TT}, this gives the statement.

For the surface $G^r_d(C)$, as the stratification in \S\ref{ss:fiber} consists of a single non-empty stratum, Corollary \ref{cor:BNG} implies $\pi_*\,T_y(G^r_d(C)) =T_y(W^r_d(C))$,
 hence $\chi_{y}(G^r_d(C)) = \chi_{y}(W^r_d(C))$.
\end{proof}

\begin{remark}
We emphasize how the stratification from \S \ref{sec:filtr} for $W^r_d(C)$ is finer than the one considered in \cite{pp}. Indeed, for this example we have two strata, while the stratification used in \cite{pp}  for the same locus consists of only one stratum, as $W^r_d(C)$ is a smooth surface.
\end{remark}

\subsection{The surface case when $r=1$} 
\label{sec:ptBNsurfpencils}
Here we consider the surface case when $r=1$.  The case $\lambda_1=\lambda_2$ being treated in \S\ref{sec:BNsurfalleqlas}, we assume here $\lambda_1>\lambda_2$.
The stratification from \S \ref{sec:filtr} consists of  the single full-dimensional stratum with $\bm{k^+}=(1,2)$, contrary to the case $\lambda_1=\lambda_2$. 
From Corollary \ref{cor:BNW}, the motivic Hirzebruch class is given by
\begin{equation}
\label{eq:Tyinterla1biggerthenla2}
T_y\left( W^{\bm{a}}_d(C,P) \right) = A_{\bm{\lambda}}.
\end{equation}
Formula \eqref{eq:TyOmegatilde} gives
\begin{eqnarray*}
T_y\left( W^{\bm{a}}_d(C,P) \right) &=&\bigg( 1+ \frac{1}{2}\left(  (\lambda_1-2)R_1+(\lambda_2-1) R_2+ 2\theta \right)(y-1)\\
&&{}+\frac{1}{24} \bigg( \left( (3 \lambda_1^2 -11 \lambda_1 +12) (y-1)^2-8y\lambda_1 \right)R_1^2\\
   && {}+2  \left(\left(3 (\lambda_1-2) \lambda_2 -3 \lambda_1 +5\right)(y-1)^2  +8y \right)R_1 R_2\\ 
   && {}+(\lambda_2-1)\left(\left(3 \lambda_2 -2\right) (y-1)^2-8y\right) R_2^2\\
   &&{}  +2 \left(\left(6 \lambda_1 -11\right)(y-1)^2 -8y\right)\theta R_1\\
   &&{} +2 \left(\left(6 \lambda_2 -5\right)(y-1)^2 -8y\right)\theta R_2\\
   &&{} +12 (y-1)^2\theta ^2 \bigg)\bigg) \left| c_{\lambda_i+j-i}\right|_{1\leq i,j\leq 2}.
\end{eqnarray*}
The degree-two operators  are:
\begin{align*}
\label{TT2}
\begin{split}
R_1^2 \left| c_{\lambda_i+j-i} \right|_{1\leq i,j \leq 2} &= g!\frac{3+\lambda_1-\lambda_2}{(\lambda_1+3)! \lambda_2!},\\
R_1 R_2 \left| c_{\lambda_i+j-i} \right|_{1\leq i,j \leq 2}  &= g!\frac{1+\lambda_1-\lambda_2}{(\lambda_1+2)! (\lambda_2+1)!}, \\
R_2^2  \left| c_{\lambda_i+j-i} \right|_{1\leq i,j \leq 2} &= g!\frac{\lambda_1-\lambda_2-1}{(\lambda_1+1)! (\lambda_2+2)!},\\
\theta\, R_1 \left| c_{\lambda_i+j-i} \right|_{1\leq i,j \leq 2} &= g!\frac{2+\lambda_1-\lambda_2}{(\lambda_1+2)! \lambda_2!},\\
\theta\, R_2 \left| c_{\lambda_i+j-i} \right|_{1\leq i,j \leq 2} &= g!\frac{\lambda_1-\lambda_2}{(\lambda_1+1)! (\lambda_2+1)!},\\
\theta^2 \left| c_{\lambda_i+j-i} \right|_{1\leq i,j \leq 2}  &= g!\frac{1+\lambda_1-\lambda_2}{(\lambda_1+1)! \lambda_2!}.
\end{split}
\end{align*}
Using these, the motivic Hirzebruch class is:

\begin{corollary}
Fix $g\geq 2$ and $\bm{a}=(0\leq a_0<a_1\leq d)$ with $\rho(g,1,d,\bm{a})=2$.
For a general smooth pointed curve $(C,P)$ of genus $g$, one has
\begin{multline*}
T_y\left( W^{\bm{a}}_d(C,P) \right) =\frac{1+\lambda_1-\lambda_2}{(\lambda_1+1)! \lambda_2!}\,\theta^{g-2} \bigg( 1+  \frac{\left(\lambda_1^2-(\lambda_2-2) \lambda_1+2\right)
   \lambda_2}{(\lambda_1+2) (\lambda_2+1)}(y-1)\theta\\
{}+\Big(\big( (2 \lambda_2 (\lambda_2+2)+1) \lambda_1^3-(\lambda_2-4) (2 \lambda_2
   (\lambda_2+2)+1) \lambda_1^2\\
   {}+(\lambda_2 (10-\lambda_2 (6
   \lambda_2+5))+3) \lambda_1-3 \lambda_2 (\lambda_2+1)^2 \big)(y-1)^2\\
   {}-(\lambda_1+2) (\lambda_2+1) \left(\lambda_1^2+4 \lambda_1-\lambda_2^2-2
   \lambda_2+3\right)y \Big)\\
   {} \frac{1}{(1+\lambda_1-\lambda_2)(\lambda_1+2)(\lambda_1+3)(\lambda_2+1)(\lambda_2+2)}\theta ^2\bigg).
\end{multline*}
\end{corollary}

We deduce the Hirzebruch $\chi_y$-genus:

\begin{corollary}
Fix $g\geq 2$ and $\bm{a}=(0\leq a_0<a_1\leq d)$ with $\rho(g,1,d,\bm{a})=2$.
For a general smooth pointed curve $(C,P)$ of genus $g$, one has
\begin{multline*}
\chi_y\left( W^{\bm{a}}_d(C,P) \right) =
 \Big(\big( (2 \lambda_2 (\lambda_2+2)+1) \lambda_1^3-(\lambda_2-4) (2 \lambda_2
   (\lambda_2+2)+1) \lambda_1^2\\
   {}+(\lambda_2 (10-\lambda_2 (6
   \lambda_2+5))+3) \lambda_1-3 \lambda_2 (\lambda_2+1)^2 \big)(y-1)^2\\
   {}-(\lambda_1+2) (\lambda_2+1) \left(\lambda_1^2+4 \lambda_1-\lambda_2^2-2
   \lambda_2+3\right)y \Big)  \frac{g!}{(\lambda_1+3)! (\lambda_2+2)!}.
\end{multline*}
The stratification in \S\ref{ss:fiber} consists of a single non-empty stratum, hence $\chi_{y}(G^{\bm{a}}_d(C,P)) = \chi_{y}(W^{\bm{a}}_d(C,P))$.
\end{corollary}

\begin{remark}
\label{rmk:piecepol}
Interestingly, the above formula for $\chi_{y}(W^{\bm{a}}_d(C,P))$ in the case $\lambda_1>\lambda_2$ (i.e., $a_1>a_0$) does not specialize to the formula for the case $\lambda_1=\lambda_2$ from \S\ref{sec:BNsurfalleqlas}, unless $y=0$. 

The discrepancy arises from the fact that the stratification in the surface case with $r=1$ consists of two strata when $\lambda_1=\lambda_2$, while there is only one stratum when $\lambda_1>\lambda_2$. Indeed, the stratum with $\bm{k^+}=(2,2)$ has codimension two when $\lambda_1=\lambda_2$, while it has codimension at least three when $\lambda_1>\lambda_2$, hence it is empty on surfaces. Consequently, the motivic Hirzebruch class $T_y\left( W^{\bm{a}}_d(C,P) \right)$ in the surface case with $r=1$ and $\lambda_1>\lambda_2$ given by \eqref{eq:Tyinter} does not specialize to the surface case with $\lambda_1=\lambda_2$ given by \eqref{eq:Tyinterla1biggerthenla2}, unless $y=0$.

This is in contrast with the case $y=0$, corresponding to the holomorphic Euler characteristic of a surface $W^{\bm{a}}_d(C,P)$, which following \cite{act}, does specialize from the case $\lambda_{i}>\lambda_{i+1}$ to the case  $\lambda_{i}=\lambda_{i+1}$, for any $i$. When $r=1$, this can be seen here since the second summand in \eqref{eq:Tyinter} vanishes for $y=0$.
\end{remark}


\bibliographystyle{abbrv}
\bibliography{Biblio}

\end{document}